\documentclass[11pt, oneside]{amsart}

\usepackage[english]{babel} 
\usepackage{graphics,color,pgf}
\usepackage{epsfig}
\usepackage[ansinew]{inputenc}
\usepackage[all]{xy}
\usepackage{hyperref}
\newdir{ >}{!/8pt/@{}*@{>}}
\usepackage{amssymb, amsmath,amsthm, mathtools, amscd}
\usepackage{mathrsfs}
\usepackage[margin=1.5in]{geometry}
\usepackage{tikz-cd}
\usetikzlibrary{matrix,arrows,decorations.pathmorphing}

\makeatletter
\newtheorem*{rep@theorem}{\rep@title}
\newcommand{\newreptheorem}[2]{%
\newenvironment{rep#1}[1]{%
 \def\rep@title{#2 \ref{##1}}%
 \begin{rep@theorem}}%
 {\end{rep@theorem}}}
\makeatother

\newtheorem{intro_thm}{Theorem}

\newtheorem{intro_prop}[intro_thm]{Proposition}
\newtheorem{lemma}{Lemma}[section]
\newtheorem{thm}[lemma]{Theorem} 
\newtheorem{prop}[lemma]{Proposition}

\theoremstyle{definition}
\newtheorem{defn}[lemma]{Definition}
\newtheorem{es}[lemma]{Example}

\theoremstyle{remark}
\newtheorem{oss}[lemma]{Remark}

\newcommand\matH{{\mathbb{H}}^n_{\mathbb{C}}}

\newcommand\matR{{\mathbb{R}}}

\newcommand\matN{{\mathbb{N}}}

\newcommand\matC{{\mathbb{C}}}

\newcommand{\parallelsum}{\mathbin{\!/\mkern-5mu/\!}}

\newcommand{\id}{\mathrm{id}}

\newcommand\calQ{{\mathcal Q}}
\newcommand\calT{{\mathcal T}}
\newcommand\calC{{\mathcal C}}

\newcommand\calB{{\mathcal B}}
\newcommand\calF{{\mathcal F}}

\newcommand\calG{{\mathcal G}}
\newcommand\calH{{\mathcal H}}
\newcommand\calY{{\mathcal Y}}
\newcommand\calP{\mathcal{P}}
\newcommand\calI{\mathcal{I}_p}
\newcommand\calIk{\mathcal I_k(p,\infty)}
\newcommand\calIpq{\mathcal I_p(p,q)}
\newcommand\calX{\mathcal{X}}

\newcommand{\pup}{\textup{PU}(p,\infty)}

\newcommand{\pupq}{\textup{PU}(p,q)}

\newcommand{\puoneinf}{\textup{PU}(1,\infty)}

\newcommand{\su}{\textup{PU}(n,1)}

\newcommand{\uppq}{\textup{U}(p,q)}

\newcommand{\upinf}{\textup{U}(p,\infty)}

\newcommand{\Hm}{\textup{\textup{H}}}
\newcommand{\Hb}{\textup{\textup{H}}_{\text{b}}}

\newcommand{\Hcb}{\textup{\textup{H}}_{\text{cb}}}
\newcommand{\Hbul}{\textup{\textup{H}}^{\bullet}}

\newcommand{\Linf}{\text{L}^{\infty}}

\newcommand{\Tbdue}{\textup{T}_b^2}
\newcommand{\Sim}{\mathcal{X} (p,\infty)}
\newcommand{\Simpq}{\mathcal{X} (p,q)}
\newcommand{\GL}{\textup{GL}(\calH)}

\newcommand{\rk}{\text{rk}}

\newcommand{\EssIma}{\text{Ess Im}}

\newcommand{\Isom}{\text{Isom}}

\newcommand{\Stab}{\text{Stab}}
\newcommand{\dime}{\text{dim}}
\newcommand{\Span}{\text{Span}}
\newcommand{\Id}{\text{Id}}
\newcommand{\rad}{\text{rad}}
\newcommand{\diam}{\text{diam}}

\begin{document}

\title[Boundary maps for cocycles into CAT(0)-spaces]{Boundary maps and reducibility for cocycles into the isometries of CAT(0)-spaces}

\author[F. Sarti]{F. Sarti}
\address{Department of Mathematics, University of Pisa, Italy}
\email{filippo.sarti@dm.unipi.it}

\author[A. Savini]{A. Savini}
\address{Department of Mathematics, University of Milano-Bicocca, Milano, Italy}
\email{alessio.savini@unimib.it}

\date{\today \; \copyright{\ F. Sarti, A. Savini.
The first author is funded by MUR through the PRIN project ``Geometry and topology of manifolds". The authors are partially supported by INdAM--GNSAGA}}

\begin{abstract}
Let $\Gamma$ be a discrete countable group acting isometrically on a measurable field $\mathbf{X}$ of CAT(0)-spaces of finite telescopic dimension over some ergodic standard Borel probability $\Gamma$-space $(\Omega,\mu)$. If $\mathbf{X}$ does not admit any invariant Euclidean subfield, we prove that the measurable field $\widehat{\mathbf{X}}$ extended to a $\Gamma$-boundary admits an invariant section. In the case of constant fields this shows the existence of Furstenberg maps for measurable cocycles, extending results by Bader, Duchesne and L\'ecureux.

When $\Gamma<\su$ is a torsion-free lattice and the CAT(0)-space is $\mathcal{X}(p,\infty)$, we show that a maximal cocycle $\sigma:\Gamma \times \Omega \rightarrow \pup$ with a suitable boundary map is finitely reducible. As a consequence, we prove an infinite dimensional rigidity phenomenon for maximal cocycles in $\mathrm{PU}(1,\infty)$. 
\end{abstract}
  
\maketitle

\section{Introduction}

Boundary maps for representations in algebraic groups were first introduced by Furstenberg \cite{furstenberg:annals,furst:articolo73} and represent a powerful tool in the investigation of rigidity phenomena. Examples of results involving such maps are Mostow rigidity \cite{mostow68:articolo} and Margulis superrigidity \cite{margulis:super}. Boundary maps gained even more importance in the context of bounded cohomology, where their application to the work by Burger and Monod \cite{burger2:articolo} has generated a prolific literature \cite{iozzi02:articolo,BIcartan,bucher2:articolo,Pozzetti,BBIborel}.

More recently, several authors focused their attention on actions on CAT($\kappa$)-spaces and the associated boundary maps. For instance, CAT(-1)-spaces have been studied by Burger and Mozes \cite{burger:mozes} and by Monod and Shalom \cite{MonShal0}. Duchesne focused first on actions on the Hermitian symmetric space $\calX(p,\infty)$ \cite{duchesne:2013} and then, together with Bader and L\'ecureux, on a general CAT(0)-space $\mathcal{X}$ of finite telescopic dimension \cite{bader:duchesne:lecureux:16}. There the authors proved the existence of a boundary map $B\rightarrow \partial\calX$ whenever the action on $\calX$ is not elementary. In this setting $B$ denotes a $\Gamma$-boundary in the sense of Bader and Furman \cite{BF14}. Such boundary can be seen as an extension both of the Furstenberg-Poisson boundary \cite{furstenberg:annals} and of the strong boundary in the sense of Burger-Monod \cite{burger2:articolo}. In the general setting studied in \cite{bader:duchesne:lecureux:16}, we lose the structure of spherical building of $\partial\calX(p,\infty)$ exploited in \cite{duchesne:2013}, but we can still rely on the rich structure of CAT(0)-spaces. For instance, a useful tool is the Euclidean de-Rham decomposition (see \cite{bridson:haefliger:13}). Additionally,  when the telescopic dimension is finite, Caprace and Lytchak \cite{caprace:lytchak:09} proved that a filtering family of closed convex subspaces of a CAT(0)-space $\calX$ have a point fixed by $\Isom(\calX)$ in the bordification $\overline{\calX}$. 
Both in \cite{duchesne:2013} and in \cite{bader:duchesne:lecureux:16} the arguments rely on the notion of \emph{measurable fields of \textup{CAT(0)}-spaces}, that were first introduced by Anderegg and Henry \cite{anderegg:henry:14} and then developed by Duchesne \cite{duchesne:2013}.

Following the line of some recent works about measurable cocycles by the authors and Moraschini \cite{savini:20,moraschini:savini,moraschini:savini:2,sarti:savini:superrigidity,sarti:savini:parametrized,savini3:articolo}, in
the first part of this paper we prove a generalization of \cite[Theorem 1.1]{bader:duchesne:lecureux:16} to measurable fields. For all the definitions concerning measurable fields we refer to Section \ref{section_measurable_fields}. 

\begin{intro_thm}\label{theorem_boundary_map}
 Let $\Gamma$ be a discrete countable group, $(\Omega,\mu)$ be an ergodic standard Borel probability $\Gamma$-space and $B$ a $\Gamma$-boundary.
 Consider a measurable field $\mathbf{X}$ of complete separable \textup{CAT(0)}-spaces of finite telescopic dimension endowed with an isometric $\Gamma$-action. If $\mathbf{X}$ does not admit any invariant Euclidean subfield, then there exists an invariant section of the boundary field $\partial \widehat{\mathbf{X}}$, where $\widehat{\mathbf{X}}$ is the extension of $\mathbf{X}$ to the boundary $B$. 
\end{intro_thm}

An immediate consequence of the previous theorem is the existence of invariant sections for the boundary of constant fields and hence the existence of boundary maps for measurable cocycles. 

\begin{intro_prop}\label{prop_boundary_map}
Let $\Gamma$ be a discrete countable group, $(\Omega,\mu)$ be an ergodic standard Borel probability space and $B$ a $\Gamma$-boundary. Consider a complete separable \textup{CAT(0)}-space $\mathcal{X}$ of finite telescopic dimension and a measurable cocycle $\sigma: \Gamma \times \Omega \rightarrow \mathrm{Isom}(\mathcal{X})$. If $\mathcal{X}$ does not any admit Euclidean subfield on $\Omega$ which is $\sigma$-invariant, then there exists a boundary map $\phi:\Gamma \times \Omega \rightarrow \partial \mathcal{X}$ for $\sigma$. 
\end{intro_prop} 

The proof of Theorem \ref{theorem_boundary_map} is based on the arguments used in \cite[Theorem 1]{bader:duchesne:lecureux:16}, where the crucial point is the measurable version of the Adam-Ballmann theorem \cite[Theorem 1.8]{duchesne:2013}. Thanks to \cite[Proposition 8.11]{duchesne:2013} we can work with minimal invariant subfields. Moreover, by applying the measurable Euclidean de-Rham decomposition \cite[Proposition 9.2, Proposition 9.3]{duchesne:2013}, we reduce ourselves to the particular case when there exists an invariant minimal family of closed convex spaces with trivial Euclidean factors.

Starting from Theorem \ref{theorem_boundary_map} we investigate the case of the constant field $\calX=\calX(p,\infty)$.
Recently, Duchesne, L\'ecureux and Pozzetti \cite{duchesne:pozzetti} proved that any maximal representation $\rho:\Gamma\rightarrow \pup$
of a lattice $\Gamma<\su$, with $n\geq2$, preserves a finite dimensional totally geodesic
Hermitian symmetric space $\calY\subset \Sim$.
Moreover, 
under the additional hypothesis of Zariski density, they ruled out the existence of any such representation for any $p\geq 1$.

Motivated by such results, we will focus our attention on measurable cocycles $\sigma: \Gamma \times \Omega\rightarrow \pup$, where $\Gamma$ is a complex hyperbolic lattice 
in $\su$ and $(\Omega,\mu)$ is an ergodic standard Borel probability $\Gamma$-space. 
We actually need to assume something more, namely the existence of a boundary map $\partial\matH\times \Omega\rightarrow \calI$, where $\calI$ denotes the set of $p$-isotropic subspaces of $\calH=\matC^{p,\infty}$. The existence of those maps will be discussed in the last section.

The main difference with the finite dimensional case \cite{Pozzetti,sarti:savini:superrigidity} is that the group $\pup$ is not algebraic in the usual meaning. The absence of such structure motivates the notions of \emph{algebraic} and of \emph{standard algebraic subgroup} given by Duchesne, Pozzetti and L\'ecureux \cite{duchesne:pozzetti}. In this way they were able to define the notion of Zariski density inside $\pup$.

The lack of an algebraic structure can be overcome, for instance, when $\sigma$ is cohomologous to a cocycle whose image is contained in a \emph{finite dimensional} algebraic subgroup.
We call such cocycles \emph{finitely reducible}.
Using the machinery of numerical invariants and maximality developed by Moraschini and the second author \cite{moraschini:savini,moraschini:savini:2}, we get a statement similar to \cite[Theorem 6.7]{duchesne:pozzetti} for cocycles.

\begin{intro_thm}\label{theorem_reducibility}
 Let $\Gamma<\su$ be a complex hyperbolic lattice with $n\geq 1$ and let $(\Omega,\mu)$ be an ergodic standard Borel probability $\Gamma$-space. 
 Consider a measurable cocycle \mbox{$\sigma: \Gamma \times \Omega\rightarrow \pup$} with $p\geq 1$ and suppose that there exists a boundary map \mbox{$\phi: \partial\matH\times \Omega\rightarrow \calI$}. If $\sigma$ is maximal, then it is finitely reducible.
\end{intro_thm}

The structure of the proof is the following.
We first refine \cite[Proposition 6.2]{duchesne:pozzetti}, namely we show that any slice of the boundary map has image essentially contained in a unique copy of $\partial\calX(p,q)$ embedded in $\partial \Sim$ for some $p\leq q\leq np$. Since such construction varies measurably, ergodicity implies that $q$ does not depend on the slice. Using the transitive action of $\pup$ (Lemma \ref{proposition_embedding}), we twist the cocycle and the boundary map 
in such a way to find a cohomologous cocycle $\sigma^f$ and a boundary map $\phi^f$ with image of the latter essentially contained in some embedding
of $\partial\calX(p,q)$ in $\partial \Sim$, so that finite reducibility follows.

It seems natural to ask whether Theorem \ref{theorem_boundary_map} provides a suitable boundary map in the context of Theorem \ref{theorem_reducibility}. 
We notice that by our first result we have an equivariant map $\partial\matH \times \Omega\rightarrow  \mathcal{I}_k$ for some $k\leq p$. In particular, for cocycles $\sigma: \Gamma \times \Omega\rightarrow \puoneinf$, since maximality implies non-elementarity, Theorem \ref{theorem_boundary_map} provides a boundary map $\partial\matH \times \Omega\rightarrow  \partial\mathbb{H}^{\infty}_{\matC}$ and, by applying Theorem \ref{theorem_reducibility} and \cite[Theorem 2]{sarti:savini:superrigidity}, we get the following version of Mostow rigidity for infinite dimensional cocycles.

\begin{intro_thm}\label{corollary_mostow}
Let $\Gamma<\su$ be a complex hyperbolic lattice with $n\geq 1$ and let $(\Omega,\mu)$ be an ergodic standard Borel probability $\Gamma$-space. 
Any maximal cocycle $\sigma :\Gamma \times \Omega \rightarrow \puoneinf$ is cohomologous to a cocycle preserving a copy of $\matH\subset \mathbb{H}^{\infty}_{\matC}$
and acting on it via the standard lattice embedding.
\end{intro_thm}

\paragraph{\textbf{Plan of the paper.}} The paper is divided into three sections.
Section \ref{section_boundary_maps} focuses on the existence of invariant boundary sections and boundary maps. After a brief introduction of basics about CAT(0)-spaces (Section \ref{section_cat0}), we define measurable fields of metric and CAT(0)-spaces. We recall the measurable Euclidean de-Rham decomposition and the measurable version of Adam-Ballmann theorem (Section \ref{section_measurable_fields}). 
Then we define boundaries and we prove Theorem \ref{theorem_boundary_map} (Section \ref{section:existence}). We conclude the section recalling the notion of boundary maps and proving Proposition \ref{prop_boundary_map}. 

Section \ref{section:reducibility} is devoted to reducibility of cocycles into the isometries of $\calX(p,\infty)$.
We first recall some notions about bounded cohomology (Section \ref{section_bounded_cohomology}). Then we introduce Hermitian symmetric spaces and we characterize embeddings 
of $\Simpq$ inside $\Sim$ (Section \ref{section_symmetric_space}).
In this context, we define the \emph{Toledo invariant} associated to a measurable cocycle, passing 
through the definition of \emph{Bergman class} and the machinery developed by \cite{moraschini:savini} about numerical invariants of measurable cocycles (Section \ref{section_khaler} and \ref{section_toledo}).
Then we move to the notion of \emph{algebraic} and \emph{finite dimensional algebraic subgroup} of $\GL$ (Section \ref{section_algebraic}) and we
finally provide the proof of Theorem \ref{theorem_reducibility} (Section \ref{section_proof_reducibility}).

We conclude with Section \ref{section_consequences}, where we prove Theorem \ref{corollary_mostow}.

\vspace{ 0.5 cm}

\paragraph{\textbf{Acknowledgments.}} We would like to thank Maria Beatrice Pozzetti and Bruno Duchesne for the enlightening conversations
and important suggestions about our work. We are also grateful to Stefano Francaviglia for his supervision and comments about the project.
Finally, we wish to thank the anonymous referees for their comments and corrections that improved the quality of this paper. 

\section{Existence of invariant boundary sections and boundary maps}\label{section_boundary_maps}

This section is devoted to prove the existence of boundary maps for measurable cocycles in the isometries of a complete separable CAT(0)-space. After a brief introduction about CAT(0)-spaces, measurable fields and boundaries in the sense of Bader and Furman \cite{BF14}, we give the proof of Theorem \ref{theorem_boundary_map}. Then we move to measurable cocycles and boundary maps to prove Proposition \ref{prop_boundary_map}

\subsection{CAT(0)-spaces}\label{section_cat0}

A metric space $(\calX,d)$ is a CAT(0)-\emph{space} if it is geodesic and for every triple of distinct points $x,y,z\in \calX$, given a point $m$ in the geodesic segment between $y$ and $z$, the following inequality holds
$$d(x,m)^2\leq \frac{1}{2} (d(x,y)^2+d(x,z)^2)-\frac{1}{4} d(y,z)^2.$$
For the purposes of this paper, CAT(0)-spaces will be always assumed to be \emph{complete} and \emph{separable}.
 
Since embedded flats into CAT(0)-spaces play an important role in the study of their geometry, we recall the following decomposition into Euclidean and non-Euclidean factors. 
Precisely, the \emph{Euclidean de-Rham decomposition} of a CAT(0)-space $\calX$ is its canonical isometric splitting into an Hilbert space $H$ and a factor $Z$ which cannot be further decomposed as a product with non-trivial Euclidean factor \cite[Theorem 6.15]{bridson:haefliger:13}. 
Moreover, for every point $x\in \calX$ the space $H$ (respectively $Z$) identifies with a unique closed convex subspaces of $\calX$ containing $x$. 
 
Given a subset $\calY\subset \calX$ of a metric space, its \emph{diameter} is defined as
$$\diam(\calY)\coloneqq\sup_{x,y\in \calY} d(x,y),$$
and $\calY$ is said to be \emph{bounded} if it has finite diameter.
A convex bounded set $\calY$ has some preferred points called \emph{circumcenters}, which are the centers of balls of minimal radius containing $\calY$. Notice that, without the assumption of convexity, one can still give the notion of circumcenter but such points may not belong to $\calY$. In the case of CAT(0)-spaces it turns out that every bounded subset has a unique circumcenter, which we call \emph{center}. An equivalent definition can be given in terms of actions of isometries. Precisely, the center of a bounded subset $\calY\subset \calX$ of a CAT(0)-space is the unique point fixed by any isometry stabilizing $\calY$. 
 
Before introducing the notion of telescopic dimension, we need the one of \emph{geometric dimension}. This concept was first introduced by Kleiner \cite{kleiner:99} in terms of the space of directions at each point, and then has been reformulated by Caprace and Lytchack \cite[Theorem 1.3]{caprace:lytchak:09} in the following way. 
If $\calX$ is a CAT(0)-space, then its geometric dimension is $\leq n$ if for each subset $\calY$ of finite diameter the following inequality holds
$$\rad(\calY)\leq \sqrt{\frac{n}{2(n+1)}}\diam (\calY)\ .$$
The number $\rad(\calY)$ is the \emph{circumradius} of $\calY$, namely the infimum of all $r>0$ such that $\calY$ is contained in some closed ball of radius $r$. 
The result by Caprace and Lytchack leads to a characterization of telescopic dimension, originally given by Kleiner \cite{kleiner:99}, that we assume here as a definition (refer to \cite{caprace:lytchak:09} for more details). 
\begin{defn}
A CAT(0)-space $\calX$ has \emph{telescopic dimension} $\leq n$ if for any $\delta >0$ there exists some constant $D>0$ such that for every bounded set $\calY$ of diameter $>D$, we have
$$\rad(\calY)\leq\left( \delta +\sqrt{\frac{n}{2(n+1)}} \right)\diam (\calY).$$
\end{defn}

As we will recall in Section \ref{section_symmetric_space}, the Hermitian symmetric space $\calX(p,\infty)$ is a CAT(0)-space of telescopic dimension $p$ \cite[Corollary 1.4]{duchesne:2013}. This implies that the visual boundary $\partial\calX(p,\infty)$ has geometric dimension $p-1$ \cite[Proposition 2.1]{caprace:lytchak:09}.

For a complete CAT(0)-space $\calX$ with finite telescopic dimension, Caprace and Lytchak proved that every filtering family of closed convex subspaces of $\calX$ either intersects at $\calX$ or at $\partial\calX$ \cite[Theorem 1.1]{caprace:lytchak:09}. 
Notice that this is equivalent to quasi-compactness of the bordification $\overline{\calX}=\calX\cup \partial\calX $ endowed with the topology defined by Monod \cite[Section 3.7]{monod:06}.
The following technical result is an example of application of \cite[Theorem 1.1]{caprace:lytchak:09}, and it turned out to be the useful in the proof of \cite[Theorem 1.1]{bader:duchesne:lecureux:16} and \cite[Theorem 1.7]{duchesne:pozzetti}. It will be exploited to prove Theorem \ref{theorem_boundary_map}.

\begin{prop}[{\cite[Proposition 2.1]{bader:duchesne:lecureux:16}}]\label{proposition:bdl}
Let $E$ be an Euclidean space and $f:E \rightarrow \matR$ be a convex function. 
If we denote by $m=\inf\{f(x)\,|\, x\in E \}$, then we have the following four possible cases:
\begin{itemize}
\item[(i)] If $m$ is not attained, then $\bigcap\limits_{\epsilon>0} \partial E_\epsilon\neq \emptyset$ where $E_{\epsilon}\coloneqq  f^{-1}((m,m+\epsilon))$ is not empty and has a center. 
\end{itemize}
If $m$ is attained, we denote by $E_m=f^{-1}(m)$ and by $E_m=F\times T$ its Euclidean de-Rham decomposition. Then one of the following holds
\begin{itemize}
\item[(ii)] $E_m$ is bounded and thus it has a center;
\item[(iii)] 	$T$ is bounded and $\partial E_m=\partial F$ is a sphere;
\item[(iv)] $T$ is not bounded and $\partial T\subset \partial E$ has radius less then $\frac{\pi}{2}$.
\end{itemize}
\end{prop}

Notice that, as mentioned in point $(iii)$, boundaries of flats are Euclidean spheres, that can be also interpreted as CAT(1)-spaces. 
In particular, boundaries of maximal flats are subcomplexes, called apartments, of the building structure of the visual boundary. We refer to \cite{abramenko08} for the general theory of such building. We only point out that the existence of circumcenters for bounded subsets \cite[Proposition 2.7]{bridson:haefliger:13} holds also in this case. 
More precisely, every subset of radius at most $\frac{\pi}{2}$ in a sphere has a center, and this property will be used in the proof of Theorem \ref{theorem_boundary_map}.

\subsection{Measurable fields and the Adam-Ballmann dichotomy}\label{section_measurable_fields}
In this section we introduce measurable fields of metric spaces together with some results that we will exploit in the next section to prove the existence of boundary maps. 
We refer to Anderegg and Henry \cite{anderegg:henry:14} for the general theory of measurable fields and to Duchesne \cite{duchesne:2013} for the measurable version of both the Euclidean de-Rham decomposition and the Adam-Ballmann dichotomy.

\begin{defn}\label{definition_measurable_field}
Given a standard probability space $(\Omega,\mu)$, a \emph{measurable field of metric spaces} on $\Omega$ is a collection of metric spaces $\textbf{\textup{X}}=\{X_{\omega}\}_{\omega\in \Omega}$ together with a countable family $\calF\subset \prod\limits_{\omega\in \Omega} X_{\omega}$
such that
\begin{itemize}
\item for all $x,y\in \calF$ the map $\omega\mapsto d_{\omega}(x_\omega,y_\omega)$ is measurable;
\item for almost every $\omega\in \Omega$, the set $\{f_{\omega}\,|\, f\in  \calF\}$ is dense in $X_\omega$.
\end{itemize}
A \emph{section} of $\textbf{\textup{X}}$ is an element $x\in \prod\limits_{\omega\in \Omega} X_{\omega}$ such that, for every $y\in \calF$ the map $\omega\mapsto d_{\omega}(x_{\omega},y_{\omega})$ is measurable.\\
If the $X_{\omega}$'s are CAT(0)-spaces, then we call $\textbf{X}$ a {measurable field of CAT(0)-spaces}. A \emph{subfield} $\textbf{\textup{Y}}$ of $\textbf{\textup{X}}$ is a collection of non-empty closed convex subsets $Y_{\omega}\subset X_{\omega}$ such that, for every section $\Omega$ of $\textbf{\textup{X}}$ the map $\omega\mapsto d_{\omega} (x_{\omega},Y_{\omega})$ is measurable.
\end{defn}

If $G$ is a locally compact group acting on a standard probability space $(\Omega,\mu)$ by preserving the measure class, we say that $\Omega$ is a \emph{Lebesgue} $G$-space. A $G$-\emph{action} on $\textbf{\textup{X}}$ is a collection $\{\sigma( g,\omega)\}_{g\in G,\omega\in \Omega}$ where
\begin{itemize}
\item for every $g\in G$ and almost every $\omega\in \Omega$, we have $\sigma(g,\omega)\in \Isom (X_{\omega},X_{g\omega})$;
\item for every $g,h\in G$ and almost every $\omega\in \Omega$, the following equality holds
\begin{equation}\label{cocycle_condition_fields}
\sigma(gh,\omega)=\sigma(g,h \omega)\sigma(h,\omega);
\end{equation}
\item for every $x,y\in \calF$, the map $(g,\omega)\mapsto d(x_{\omega},\sigma (g,g^{-1} \omega) y_{g^{-1}\omega})$ is measurable.
\end{itemize}

\begin{es}\label{ex:constant:field}
Given a standard Borel probability space $(\Omega,\mu)$ and a complete separable metric space $\mathcal{X}$, we can build a measurable field $\textbf{X}$ by setting $X_{\omega}\coloneqq \mathcal{X}$ and taking as fundamental family the collection 
$\{f^x \}_{x \in \mathcal{X}_0}$ where 
$\mathcal{X}_0\subset \mathcal{X}$ is a countable dense subset and 
$f^x_{\omega}\coloneqq 
  x $ for every $\omega\in\Omega$. 

Given a locally compact group $G$ such that $(\Omega,\mu)$ is a Lebesgue $G$-space, a $G$-action $\{\sigma( g,\omega)\}_{g\in G,\omega\in \Omega}$ on $\textbf{X}$ boils down to a map $\sigma:G\times \Omega\to \Isom (\mathcal{X})$ satisfying Equation \ref{cocycle_condition_fields} and such that the function
$$
(g,\omega) \mapsto d(x,\sigma(g,g^{-1}\omega)y)
$$
is measurable for all $x,y \in \mathcal{X}_0$. 
\end{es}

\begin{es}\label{ex:extended:field}
Let $G$ be a locally compact group and consider a Lebesgue $G$-space $(\Omega,\mu)$. Given another Lebesgue $G$-space $(\Theta,\nu)$, the product space $(\Omega \times \Theta, \mu \otimes \nu)$ is again a Lebesgue $G$-space with respect to the diagonal action. Consider a measurable field $\mathbf{X}=\{\mathcal{X}_\omega\}_{\omega \in \Omega}$ on $\Omega$ of CAT(0)-spaces of finite telescopic dimension with an isometric $G$-action $\{\sigma(g,\omega)\}_{g \in G, \omega \in \Omega}$. We can consider the \emph{extension} $\widehat{\mathbf{X}}$ of the field $\mathbf{X}$ to the space $\Theta$ in the following way:
\begin{itemize}
\item We set $\widehat{\mathcal{X}}_{(\omega,\theta)}:=\mathcal{X}_\omega$ for $(\omega,\theta) \in \Omega \times \Theta$. 
\item Any section $x_\omega$ of $\mathbf{X}$ can be changed into a section of $\widehat{\mathbf{X}}$ by setting $\widehat{x}_{(\omega,\theta)}:=x_\omega$. 
\item We define an isometric $G$-action on $\widehat{\mathbf{X}}$ by setting $\{ \widehat{\sigma}(g,\omega,\theta):=\sigma(g,\omega) \}$ for all $g \in G, \omega \in \Omega, \theta \in \Theta$.  
\end{itemize}
\end{es}

A $G$-action $\{\sigma( g,\omega)\}_{g\in G,\omega\in \Omega}$ on a measurable field $\textbf{\textup{X}}$ of CAT(0)-spaces induces a natural $G$-action on both sections and subfields. Similarly, we have an induced action on the boundary field $\partial \mathbf{X}$ of $\mathbf{X}$. The latter is obtained by collecting the boundaries of each $\mathcal{X}_\omega$ and by controlling measurability via the Busemann functions. The $G$-action on $\mathbf{X}$ induces an action on the sections of $\partial \mathbf{X}$ given by $(g\xi)_\omega=(\sigma(g,g^{-1}\omega)\xi_{g^{-1}\omega})$. Furthermore, a subfield $\textbf{\textup{Y}}\subset \textbf{\textup{X}}$ is \emph{minimal} if its invariant under the $G$-action and it does not contain a proper invariant subfield. 

As proved by Caprace and Lytchak \cite[Proposition 1.8]{caprace:lytchak:09}, 
any isometric action of a locally compact group on a complete CAT(0)-space of finite telescopic dimension either has a fixed point in the boundary or admits an invariant non-empty closed convex subset which is minimal with respect to inclusion. This allows to reduce the investigation of existence of boundary maps to minimal actions (see \cite[Theorem 1.7]{duchesne:2013} and \cite[Theorem 1.1]{bader:duchesne:lecureux:16}).
The following result can be seen as the generalization of  \cite[Proposition 1.8]{caprace:lytchak:09} to measurable fields.

\begin{prop}[{\cite[Proposition 8.11]{duchesne:2013}}]\label{proposition_minimal_subfield}
Let $G$ be a locally compact group and let $(\Omega,\mu)$ be a Lebesgue $G$-space. Suppose that $\textbf{\textup{X}}$ is a measurable field on $\Omega$ of \textup{CAT(0)}-spaces of finite telescopic dimension, $G$ acts on $\textbf{\textup{X}}$ and $\Omega$ is ergodic. Then either there exists a minimal invariant subfield of $\textbf{\textup{X}}$ or there exists an invariant section of $\partial \textbf{\textup{X}}$.
\end{prop}

A second construction that we will use is the extension of the Euclidean de-Rham decomposition for measurable fields of CAT(0)-spaces.
\begin{prop}[{\cite[Proposition 9.2]{duchesne:2013}}]\label{proposition_euclidean_decomposition}
Let $G$ be a locally compact group and let $(\Omega,\mu)$ be a Lebesgue $G$-space. Let $x$ be a section of a measurable field $\mathbf{X}$ on $\Omega$ of \textup{CAT(0)}-spaces of finite telescopic dimension. Suppose $G$ acts minimally on $\mathbf{X}$ via $\sigma=\{\sigma(g,\omega)\}_{g\in G,\omega\in \Omega}$ and assume that the action $\Omega$ is ergodic. There exists $n\in \matN$ and two subfields $\textbf{\textup{E}}$ and $\textbf{\textup{Y}}$ of $\textbf{\textup{X}}$ containing $x$ such that $\textbf{\textup{X}}=\textbf{\textup{E}}\times\textbf{\textup{Y}}$ and $E_{\omega}\cong \matR^n$ for almost every $\omega\in \Omega$. Moreover, $\textbf{\textup{E}}$ is maximal for those properties.

If $y$ is another section of $\textbf{\textup{X}}$ and $\textbf{\textup{X}}=\textbf{\textup{E'}}\times\textbf{\textup{Y'}}$ is another such decomposition associated to $y$ then for almost every $\omega\in \Omega$, the projections $\pi_{E_{\omega}|E_{\omega}^{'}}$ and $\pi_{Y_{\omega}|Y_{\omega}^{'}}$ are isometries. As a consequence the $G$-action on $\textbf{\textup{X}}$ splits as 
$$\sigma( g,\omega)=\sigma_{\textbf{\textup{E}}}( g,\omega)\times \sigma_{\textbf{\textup{Y}}}( g,\omega)$$ where $\{\sigma_{\textbf{\textup{E}}}( g,\omega)\}_{g\in G,\omega\in \Omega}$ and $\{\sigma_{\textbf{\textup{Y}}}( g,\omega)\}_{g\in G,\omega\in \Omega}$ are respectively actions on $\textbf{\textup{E}}$ and $\textbf{\textup{Y}}$.

\end{prop}

The last preliminary result that we recall is the measurable version of the Adam-Ballmann dichotomy \cite{adam:ballmann:88}. In order to state it, we need to recall the definition of \emph{amenable space} due to Zimmer \cite{Zimmer}. 
Given a locally compact second countable group $G$, a Lebesgue $G$-space $(\Omega,\mu)$ is \emph{$G$-amenable} if 
for every $G$-action on a separable measurable field $\textbf{E}$ of Banach spaces over $\Omega$ and every $G$-invariant subfield $\textbf{K}$ of weakly compact subsets of the unit balls of $\textbf{E}^*$, there exists an invariant section of $\textbf{K}$.

We conclude the section with the following 
 
\begin{thm}[{\cite[Theorem 1.8]{duchesne:2013}}]\label{theorem_adam_Ballmann}
Let $G$ be a locally compact second countable group and $(\Omega,\mu)$ be a Lebesgue $G$-space which is ergodic and amenable. Let $\mathbf{X}$ be a measurable field on $\Omega$ of complete \textup{CAT(0)}-spaces of finite telescopic dimension. If $G$ acts on $\textbf{\textup{X}}$, then either there is an invariant section of the boundary field $\partial\textbf{\textup{X}}$ or there exists an invariant Euclidean subfield of $\textbf{\textup{X}}$.
\end{thm}

\subsection{Existence of invariant sections for extended fields}\label{section:existence}

In this section we prove Theorem \ref{theorem_boundary_map}. We need first to recall the definition of boundary in the sense of Bader and Furman \cite{BF14}. 

\begin{defn}\label{def:rel:iso:ergodic}
Let $G$ be a locally compact second countable group. A \emph{fiberwise isometric $G$-action} on a measurable map $p:M\rightarrow T$ between standard Borel spaces is a $G$-invariant Borel map 
  $d:M\times_p M\rightarrow \matR_{\geq 0}$ such that any fiber $p^{-1}(t)\subset M$ endowed with the induced metric $d_{|p^{-1}(t)\times p^{-1}(t)}$ is a separable metric space on which $G$ acts in a compatible way, namely
  $$d(g x,g y)= d(x,y)\ ,$$
  for every $g\in G$ and every $x,y\in M$ with $p(x)=p(y)$.

A map $q:X\rightarrow Y$ between Lebesgue $G$-spaces is \emph{relatively metrically ergodic} if for any fiberwise isometric $G$-action on a measurable map $p:M\rightarrow T$ between standard Borel spaces and measurable $G$-equivariant maps $f:X\rightarrow M$ and $h:Y\rightarrow T$ there exists a measurable $G$-equivariant map $\psi:Y\rightarrow M$ such that the following diagram commutes
\begin{center}
\begin{tikzcd}
X \arrow{r}{f}\arrow{d}{q}     & M \arrow{d}{p}   \\
Y \arrow{r}{h}\arrow[dotted]{ru}{\psi}    & T.
\end{tikzcd}
\end{center}
\end{defn}

 \begin{defn}[{\cite[Definition 2.3]{BF14}}]\label{def_boundary}
 Let $\Gamma$ be a locally compact and second countable group. A $\Gamma$\emph{-boundary} is an amenable Lebesgue $\Gamma$-space $(B,\nu)$ such that the projections $\pi_1:B\times B\rightarrow B$ and $\pi_2:B\times B\rightarrow B$ on the first and second factor, respectively, are relatively metrically ergodic.
 \end{defn}

 \begin{oss} The notion of $\Gamma$-boundary contains other versions of boundaries.
 \begin{itemize} 
 \item[(i)] The Furstenberg-Poisson boundary of a locally compact second countable group \cite{furstenberg:annals} turns out to be a boundary in the sense of Definition \ref{def_boundary} \cite[Theorem 2.7]{BF14}.
 \item[(ii)] Suppose that $\Gamma<H$ is a lattice into a connected semi-simple Lie group of non-compact type. Given a minimal parabolic subgroup $P<H$, the quotient $H/P$ is the Furstenberg-Poisson boundary for $\Gamma$ and hence a $\Gamma$-boundary by \cite[Theorem 2.3]{BF14}. 
\item[(iii)] In general, a $\Gamma$-boundary is a strong $\Gamma$-boundary in the sense of Burger and Monod \cite{burger2:articolo},\cite[Remarks 2.4]{BF14}.
 \end{itemize}
 \end{oss}
 
Since the arguments in the proof of Theorem \ref{theorem_boundary_map} strongly rely on the objects introduced in Section \ref{section_measurable_fields}, we recall the following result, due to Duchesne, L\'ecureux and Pozzetti.

\begin{lemma}{\cite[Lemma 4.11]{duchesne:pozzetti}}\label{lemma_dlp}
Let $\Gamma$ be a countable group and let $\textbf{\textup{X}}$ be a measurable field over a Lebesgue $\Gamma$-space $(\Omega,\mu)$. Then there exists a full-measure subset $\Omega_0\subset\Omega$, a standard Borel structure on $X\coloneqq \bigsqcup\limits_{\omega
\in \Omega_0} X_{\omega}$ and a Borel map $p:X \rightarrow \Omega_0$ that admits a fiberwise isometric $\Gamma$-action. Moreover, $p^{-1}(\omega)$ is $X_{\omega}$ with the metric $d_{\omega}$.
\end{lemma}

\begin{proof}[Proof of Theorem \ref{theorem_boundary_map}]

Without loss of generality we can suppose that the $\Gamma$-action on the measurable field $\mathbf{X}:=\{\mathcal{X}_\omega\}_{\omega \in \Omega}$ is minimal.
 In fact by Proposition \ref{proposition_minimal_subfield} either we have a minimal subfield $\textbf{\textup{Y}}\subset\textbf{\textup{X}}$ on $\Omega$ or there exists an invariant section of $\partial \textbf{\textup{X}}$. In the second case the same section can be viewed as a section of the boundary field $\partial \widehat{\mathbf{X}}$ of the extension of $\mathbf{X}$ on $B$, and hence we would conclude. 

According to Proposition \ref{proposition_euclidean_decomposition}, we consider the Euclidean de-Rham decomposition $\textbf{\textup{Y}}=\textbf{\textup{F}}\times \textbf{\textup{Z}}$ and we denote by $\sigma_{\textbf{\textup{Z}}}$ and $\sigma_{\textbf{\textup{Y}}}$ the $\Gamma$-actions induced respectively on $\textbf{\textup{Z}}$ and $\textbf{\textup{Y}}$.
We claim that the $\Gamma$-action $\sigma_{\textbf{\textup{Z}}}$ on $\textbf{\textup{Z}}$ is minimal. By contradiction, assume that it is not. Thus, by Proposition \ref{proposition_minimal_subfield}, there exists a minimal invariant subfield $\textbf{W}\subset \textbf{\textup{Z}}$ whose product with $\textbf{F}$ is a strict subfield of $\textbf{\textup{F}}\times \textbf{\textup{Z}}=\textbf{\textup{Y}}$, contradicting the minimality of $\textbf{\textup{Y}}$. Moreover, any $\sigma_{\mathbf{Z}}$-invariant Euclidean subfield of $\mathbf{Z}$ would produce an invariant Euclidean subfield of $\mathbf{X}$ and this would contradict the hypothesis. Notice also that the boundary $\partial Z_\omega$ is contained in $\partial Y_\omega$ and a fortiori in $\partial \mathcal{X}_\omega$. 

Now we have a measurable field $\mathbf{Z}$ on $\Omega$ which is minimal and it does not admit any invariant Euclidean subfield. Following Example \ref{ex:extended:field}, we consider the extension $\widehat{\mathbf{Z}}$ of the field $\mathbf{Z}$ to $B$. By \cite[Proposition 2.4]{MonShal0} the spaces $B\times \Omega$ and $B\times B\times \Omega$ are ergodic $\Gamma$-spaces. By \cite[Proposition 4.3.4]{zimmer:libro}  $B\times \Omega$ is also $\Gamma$-amenable. 
In this context we apply \cite[Theorem 1.8]{duchesne:2013} to $\widehat{\mathbf{Z}}$ and we have two possible cases: either there exists a section of $\partial \widehat{\mathbf{Z}}$ or there exists an invariant Euclidean subfield $\textbf{\textup{E}}\subset \widehat{\mathbf{Z}}$.

We consider the distance map 
$$d:B\times B\times \Omega\rightarrow \matR\,,\;\;\; (\xi_1,\xi_2,\omega)\mapsto d(E_{\xi_1,\omega}, E_{\xi_2,\omega})\coloneqq \inf\limits_{y\in E_{\xi_1,\omega}} d(y,E_{\xi_2,\omega})$$ 
where $\textbf{\textup{E}}=\{E_{\xi,\omega}\}_{(\xi,\omega)\in B\times \Omega}$.
Following \cite{bader:duchesne:lecureux:16}, we have four possible cases, and by ergodicity one of them must happen almost surely.  For the same reason, the distance map is essentially equal to some value, say $d_0$, for almost every $\omega\in \Omega$ and $\xi_1,\xi_2\in B$.

\emph{Case (i):} Suppose that $d_0$ it
is not attained for almost every $\omega\in \Omega$ and $\xi_1,\xi_2\in B$. Hence for almost every $\omega\in \Omega$ and $\xi_1\in B$ we can define the 
subspaces
$$E^n_{\xi_1,\xi_2,\omega}\coloneqq \left\{ y\in E_{\xi_1,\omega}\, | \, d(y,E_{\xi_2,\omega})< d_0+\frac{1}{n} \right\}$$
which are nested subspaces of $E_{\xi_1,\omega}$. 
By \cite[Proposition 8.10]{duchesne:2013} we have a $\sigma$-equivariant map 
$$\psi:B\times B\times \Omega\rightarrow  \partial \textbf{\textup{E}}\, .$$
To apply correctly \cite[Proposition 8.10]{duchesne:2013} we are considering the extended field $\{E'_{\xi_1,\xi_2,\omega}\}_{(\xi_1,\xi_2,\omega)\in B\times B\times \Omega}$ such that $E'_{\xi_1,\xi_2,\omega}=E_{\xi_1,\omega}$ for every $\omega\in \Omega$ and $\xi_1,\xi_2\in B$.
It follows directly from Lemma \ref{lemma_dlp} that the projection $p$ of (a full-measure subset of) $\partial \mathbf{E}$ on $ B\times \Omega$ has a $\Gamma$-fiberwise isometric action. The metric structure on each fiber is the same as the one given in \cite{bader:duchesne:lecureux:16}. Thus we can apply relative metric ergodicity to the following diagram 
\begin{center}
\begin{tikzcd}
B\times B \arrow{r}{\Psi}\arrow{d}{\pi_1}     & \mathrm{L}^0(\Omega,\partial \textbf{\textup{E}}) \arrow{d}{p_{\Omega}}   \\
B \arrow{r}{j} & \mathrm{L}^0(\Omega,  B\times \Omega).
\end{tikzcd}
\end{center}

Here $\mathrm{L}^0(\Omega,  \cdot )$ denotes the space of measurable sections identified $\mu$-almost everywhere with the standard measurable structure coming from the topology of convergence in measure \cite[Section 4.4]{Wheeden},\cite[Notation 2.4]{fisher:morris:whyte}, $\Psi$ is defined by $\Psi(\xi_1,\xi_2)(x)\coloneqq \psi(x,\xi_1,\xi_2)$, $j$ is given by $j(\xi)(x)\coloneqq (\xi,\omega)$, $\pi_1$ is the projection on the first factor and $p_{\Omega}$ is defined as $p_{\Omega}(f)(x)\coloneqq p(f(\omega))$. The function $p_{\Omega}$ can be equipped with a fiberwise isometric $\Gamma$-action obtained by the one on $p$ by integrating along $\Omega$ the functions in each fiber (see \cite[Theorem 1]{sarti:savini:superrigidity}).

By relative metric ergodicity we have a lifting $B\rightarrow \mathrm{L}^0(\Omega,\partial \textbf{\textup{E}})$, thus
$\Psi$ does not depend on the second factor. Hence we have a $\sigma$-invariant map $B\times \Omega\rightarrow \partial \textbf{\textup{E}}\subset \partial \widehat{\mathbf{Z}}$, whose existence is ruled out by the dichotomy of Theorem \ref{theorem_adam_Ballmann}.

We can suppose that the distance $d_{\xi_1,\xi_2,\omega}$ is attained almost surely and we define the non-empty subsets
$$W_{\xi_1,\xi_2,\omega}\coloneqq \{ w\in E_{\xi_1,\omega}\,|\, d(w,E_{\xi_2,\omega})=d_0\}\subset E_{\xi_1,\omega}.$$

\emph{Case (ii):}
If the $W_{\xi_1,\xi_2,\omega}$'s are bounded, we can associate to any such subset its circumcenter $c_{\xi_1,\xi_2,\omega}$. The map 
$$\psi:B\times B\times \Omega\rightarrow  \textbf{\textup{E}}\,,\;\;\; \psi(\xi_1,\xi_2,\omega)\coloneqq c_{\xi_1,\xi_2,\omega}$$
is $\sigma$-equivariant, and by applying twice the relative metric ergodicity we obtain a map 
$\psi: \Omega\rightarrow\textbf{\textup{E}}$ such that 
$$\psi(\gamma \omega)=\sigma(\gamma,\omega) \psi (\omega).$$
Since points are 0-dimensional flats, this contradicts the hypothesis on $\mathbf{X}$.

Thus the $W_{\xi_1,\xi_2,\omega}$'s are not bounded and we can consider their Euclidean de-Rham decomposition 
$$W_{\xi_1,\xi_2,\omega}=F_{\xi_1,\xi_2,\omega}\times T_{\xi_1,\xi_2,\omega},$$
where the $F_{\xi_1,\xi_2,\omega}$'s are maximal Euclidean factors. 

\emph{Case (iii):} If $T_{\xi_1,\xi_2,\omega}$ is not bounded, as in case (i) we realize a map 
$$\psi:B\times B\times \Omega\rightarrow \partial \textbf{\textup{T}}\,,\;\;\; \psi(\xi_1,\xi_2,\omega)\coloneqq c_{\xi_1,\xi_2,\omega}$$
where $c_{\xi_1,\xi_2,\omega}$ is the center of $\partial T_{\xi_1,\xi_2,\omega}$ and 
$\textbf{\textup{T}}$ denotes the measurable field given by the $T_{\xi_1,\xi_2,\omega}$'s.
Notice that $c_{\xi_1,\xi_2,\omega}$ can be defined thanks to Proposition \ref{proposition:bdl}(iv).
Using the same arguments of case (i), we get a contradiction.

\emph{Case (iv):} Finally, if the $T_{\xi_1,\xi_2,\omega}$'s are bounded we consider a subfield $\textbf{\textup{E'}}$ of $\textbf{\textup{E}}$ whose sheets are defined as follows 
$$E'_{\xi_1,\xi_2,\omega}\coloneqq F_{\xi_1,\xi_2,\omega}\times \{t_{\xi_1,\xi_2,\omega}\}$$
for every $\omega\in \Omega$ and $\xi_1,\xi_2\in B$, where $t_{\xi_1,\xi_2,\omega}$ is the circumcenter of 
$T_{\xi_1,\xi_2,\omega}$.
The same argument used in \cite{bader:duchesne:lecureux:16} shows that in fact $E_{\xi_1,\xi_2,\omega}'=E_{\xi_1,\omega}$ for almost every $\omega\in \Omega$ and $\xi_1,\xi_2\in B$. Moreover, $E_{\xi,\omega}$ and $E_{\xi',\omega}$ are \emph{parallel}
for almost every $\omega\in \Omega$ and almost every $\xi,\xi'\in B$. Recall that two Euclidean subspaces are parallel if the restriction on the first one of the distance to the second one is constant and vice-versa (in this context, the Sandwich Lemma \cite[Exercise II.2.12(2)]{bridson:haefliger:13} guarantees that their convex hull splits isometrically as $\mathbb{R}^n \times  [0,d_0]$, for some $n$). 
We use the notation
$$E_{\xi,\omega} \parallelsum E_{\xi',\omega}\,.$$ 


By Fubini's theorem there exists an element $\xi_0\in B$ and a full-measure subset $\Delta \subset B\times \Omega$ such that 
$$E_{\xi,\omega}\parallelsum E_{\xi_0,\omega}$$ 
for every $(\xi,\omega)\in \Delta$.
We denote by $\Delta^{\Gamma}= \bigcap\limits_{\gamma\in \Gamma} \gamma\Delta$ which is still of full-measure since $\Gamma$ is countable. 
We consider the set 
$$C_{\omega} \coloneqq \textup{ convex hull } (\{E_{\xi,\omega}\}_{(\xi,\omega)\in \Delta^{\Gamma}}),$$
that can be decomposed into Euclidean de-Rham factors $E_{\omega}\times T_{\omega}$ such that 
\begin{equation}\label{equation_parallel}
 E_{\omega} \parallelsum E_{\xi,\omega} \parallelsum E_{\xi_0,\omega}
\end{equation}
for every $(\xi,\omega	)\in \Delta^{\Gamma}$.
Moreover, for almost every $\omega\in \Omega$ and $\gamma\in \Gamma$, we have
\begin{align*}
\sigma_{\textbf{\textup{Z}}} (\gamma,\omega)C_{\omega} &=\textup{ convex hull }(\sigma_{\textbf{\textup{Z}}}(\gamma,\omega) E_{\xi,\omega})_{(\xi,\omega)\in \Delta^{\Gamma}}\\
&=\textup{ convex hull }(E_{\gamma \xi,\gamma \omega})_{(\xi,\omega)\in \Delta^{\Gamma}}\\
&= \textup{ convex hull }(E_{ \xi,\gamma \omega})_{(\xi,\gamma \omega)\in \Delta^{\Gamma}}=C_{\gamma \omega}\,,
\end{align*}
where to pass from the first line to the second one we used the fact that $\textbf{E}$ is a subfield of $\textbf{Z}$ and to pass from the second line to the third one we exploited the action on $\Delta^{\Gamma}$. 
Now, by the minimality of $\textbf{\textup{Z}}$
we must have $C_{\omega}=Z_{\omega}$ for almost every $\omega\in \Omega$
and since $Z_{\omega}$ has trivial Euclidean factor, by Equation \eqref{equation_parallel} we have 
$$\dime (E_{\xi,\omega})=0$$
for every $(\xi,\omega)\in \Delta^{\Gamma}$.
Hence we have a section $B\times \Omega\rightarrow \textbf{Z}$ and, by the same argument used in \emph{case (ii)}, we have a contradiction.
\end{proof}

\subsection{Measurable cocycles and boundary maps}\label{section_cocycles}

In this section we finally prove Proposition \ref{prop_boundary_map}. We will first need a short introduction to measurable cocycles. We refer to \cite{moraschini:savini,moraschini:savini:2,sarti:savini:superrigidity} for further details. We will assume that $G$ is a locally compact and second countable group endowed with its Haar measurable structure, $H$ is a topological group endowed with its Borel $\sigma$-algebra and
$(\Omega,\mu)$ be a standard Borel probability space endowed with a measure preserving $G$-action. 

\begin{defn}\label{cocycle_definition}
 A \emph{measurable cocycle} is a measurable function $\sigma: G\times \Omega\rightarrow H$ such that
 \begin{equation*}
 \sigma(g_1g_2,\omega)=\sigma(g_1,g_2\omega)\sigma(g_2,\omega)
 \end{equation*}
 holds for almost every $g_1,g_2\in G$ and for almost every $\omega\in \Omega$. 
\end{defn}

For the reader who is confident with measured groupoid theory, measurable cocycles are almost representations of the measured groupoid $G \times \Omega$ with values in $H$. The notion of homotopic representations can be rephrased as follows.

\begin{defn}\label{twisted_cocycle_definition}
 Let $\sigma_1,\sigma_2:G\times \Omega \rightarrow H$ be two measurable cocycles, let 
 $f:\Omega\rightarrow H$ be a measurable map and denote by $\sigma_1^f$ the cocycle defined as
 $$\sigma_1^f(g,\omega)\coloneqq f(g\omega)^{-1}\sigma_1(g,\omega)f(\omega)$$
 for every $g\in G$ and almost every $\omega\in \Omega$. The cocycle $\sigma^f_1$ is the \emph{$f$-twisted cocycle associated to $\sigma_1$}. 
 We say that $\sigma_1$ is \emph{cohomologous} to $\sigma_2$ if there exists a measurable map $f$
 such that $\sigma_2=\sigma_1^f$.
\end{defn}

\begin{es} \label{ex:measurable:cocycles}.
Even if there are plenty of examples of measurable cocycles in different areas, we recall the following ones, since they play a predominant role in our results. 
\begin{itemize}
\item[(i)] For any standard Borel probability $G$-space, a homomorphism $\rho:G\rightarrow H$ between a locally compact and second countable group $G$ and a topological group $H$ defines a cocycle as follows
$$\sigma_{\rho}:G\times \Omega\rightarrow H\, , \;\;\;\; \sigma_{\rho}(g,\omega)\coloneqq \rho(g).$$
Conjugated representations gives cohomologous cocycles in the sense of Definition \ref{twisted_cocycle_definition}.

\item[(ii)] Let $\mathcal{X}$ be a CAT(0)-space of finite telescopic dimension. Fix a discrete countable group $\Gamma$ and a standard Borel probability $\Gamma$-space. We can consider the constant field $\mathbf{X}=\{ \mathcal{X} \}_{\omega \in \Omega}$ on $\Omega$. Following Example \ref{ex:constant:field}, we know that an isometric $\Gamma$-action boils down to a function $\sigma:\Gamma \times \Omega \rightarrow \mathrm{Isom}(\mathcal{X})$ which satisfies Equation \ref{cocycle_condition_fields} and is measurable with respect to the topology on $\mathrm{Isom}(\mathcal{X})$ induced by the family of pseudometrics $(g,h)\mapsto (gx,hx)$, where $x \in \mathcal{X}_0$ varies in a countable dense subset of $\mathcal{X}$.

Viceversa, any cocycle $\sigma:\Gamma \times \Omega \rightarrow \mathrm{Isom}(\mathcal{X})$ which is measurable with respect to the Borel structure induced by the previous pseudometrics on $\mathrm{Isom}(\mathcal{X})$ gives rise to an isometric $\Gamma$-action on the constant field $\mathbf{X}$. 

\end{itemize}
\end{es}

Now we are ready to give the definition of boundary map. 
\begin{defn}\label{def_boundary_map}
Let $\Gamma$ be a discrete countable group, $(\Omega,\mu)$ be a standard Borel probability $\Gamma$-space, $H$ a topological group and $Y$ a measurable $H$-space. Consider a measurable cocycle $\sigma:\Gamma \times \Omega \rightarrow H$. For any $\Gamma$-boundary $B$, a \emph{boundary map} is a measurable map
$$\phi:B \times \Omega \rightarrow Y 
$$
which is \emph{$\sigma$-equivariant}, namely
$$
\phi(\gamma b,\gamma \omega)=\sigma(\gamma,\omega)\phi(b,\omega) 
$$
for every $\gamma \in \Gamma$ and almost every $b \in B, \omega\in \Omega$. 
\end{defn}

We are finally ready to prove the existence of boundary maps. 

\begin{proof}[Proof of Proposition \ref{prop_boundary_map}]
We consider the constant field $\mathbf{X}:=\{\mathcal{X}\}_{\omega \in \Omega}$ given by Example \ref{ex:constant:field}. By Example \ref{ex:measurable:cocycles}(ii) the measurable cocycle $\sigma:\Gamma \times \Omega \rightarrow \mathrm{Isom}(\mathcal{X})$ induces an isometric $G$-action on $\mathbf{X}$. Since we assumed by hypothesis that there are no invariant Euclidean subfield of $\mathbf{X}$, we can apply Theorem \ref{theorem_boundary_map} to obtain an invariant section of the boundary field $\partial \widehat{\mathbf{X}}$, where $\widehat{\mathbf{X}}$ is the extension of $\mathbf{X}$ to $B$. By \cite[Lemma 3.10]{bader:duchesne:lecureux:16} this is equivalent to a boundary map $\phi:B \times \Omega \rightarrow \partial \mathcal{X}$. 
\end{proof}

\begin{oss}\label{remark_boundary_fiinite_dimensional}
Fix positive integers $n$ and $p\leq q$.
In the setting of Proposition \ref{prop_boundary_map}, suppose that $\Gamma$ is a complex hyperbolic lattice in $\textup{PU}(n,1)$ and $\calX(p,q)$ denotes the Hermitian symmetric space associated to the group $\textup{SU}(p,q)$ (see Section \ref{section_symmetric_space}). By Proposition \ref{prop_boundary_map} we have a boundary map $\phi:\partial\mathbb{H}^n\times \Omega\rightarrow \partial \calX(p,q)$.
 By the ergodicity of $B \times \Omega$, we have that $\phi$ takes values in the set of isotropic $k$-dimensional subspace in the boundary $\partial \calX(p,q)$, for some $k\leq p$. 
 To see this, for each pair $(\xi,\omega)\in\partial\matH\times \Omega$ one can take the smallest cell in the spherical building of $\partial\calX(p,q)$ which contains $\phi(\xi,\omega)$, that corresponds to a totally isotropic flag of $\matC^{p,q}$ (see \cite{duchesne:2013}). By ergodicity the type of this flag must be the same for almost every pair in $\partial\matH\times \Omega$ and by taking the maximal isotropic spaces of any flag we get the desired map. 
 If we assume that $\sigma$ is Zariski dense, the same argument in \cite[Theorem 1.7]{duchesne:pozzetti} show that $k=p$, namely the target is the Shilov boundary of $\calX(p,q)$. 

Now, since Zariski density implies non-elementarity, this gives an alternative proof of \cite[Theorem 1]{sarti:savini:superrigidity}.
\end{oss}

\section{Finite reducibility}\label{section:reducibility}

In the second part of this paper we study cocycle actions on the Hermitian symmetric space $\calX(p,\infty)$. 
We first give a brief overview about bounded cohomology in order to define the \emph{Toledo invariant} and \emph{maximal cocycles}. Then, we recall basic facts about $\calX(p,\infty)$ and we parametrize the space of embeddings of $\calX(p,q)$ inside $\calX(p,\infty)$. Later, we characterize \emph{finite algebraic subgroups}, and we use this notion to define \emph{finite reducibility}. 
Finally we prove Theorem \ref{theorem_reducibility}.

\subsection{Bounded cohomology} \label{section_bounded_cohomology}
Let $G$ be a locally compact and second countable group and let $E$ be a Banach $G$-module 
(namely a Banach space together with an action of $G$ by isometries). Continuous bounded cohomology is usually defined via the complex of continuous bounded functions on $G$. Burger-Monod \cite{burger2:articolo} showed that more generally we can consider the cohomology of any strong resolution of $E$ by relatively injective $G$-modules. More precisely, we have that the 
\emph{continuous bounded cohomology} of $G$
with coefficients in $E$ is the cohomology of the $G$-invariant vectors of any such resolution $(E^{\bullet},\delta^{\bullet})$, namely
$$\Hcb^k(G;E)\cong \Hm^k((E^{\bullet})^G,\delta^{\bullet}).$$

If $E$ is the dual of some Banach $G$-module endowed with the weak-* topology
and assuming that $G$ is a semisimple Lie group of non-compact type, we can
define the cochain complex of essentially bounded weak-* measurable functions on the Furstenberg boundary $B(G)$, denoted by
$(\Linf_{\textup{w}^\ast}(B(G)^{\bullet+1};E),\delta^{\bullet})$, where $\delta^{\bullet}$ is the standard homogeneous coboundary operator.
Since the previous complex can be completed to a strong resolution of $E$ by relatively injective modules, we have an isomorphism
\begin{equation}\label{eq_bm}
\Hcb^k(G;E)\cong \Hm^k(\Linf_{\textup{w}^\ast}(B(G)^{\bullet+1};E)^G,\delta^{\bullet})
\end{equation}
for any $k\geq 0$. By \cite[Corollary 1.5.3]{burger2:articolo} the isomorphism is actually isometric, that is it preserves the natural seminormed structures on those spaces. 

If we consider the complex of bounded weak-* measurable functions on a measurable $G$-space $X$,
denoted by $(\calB^{\infty}_{\textup{w}^\ast}(X^{\bullet+1};E),\delta^{\bullet})$,
we obtain only a strong resolution of $E$.
Nevertheless, Burger and Iozzi \cite{burger:articolo} showed that there exists a canonical nontrivial map
\begin{equation}\label{canonical_map}
\mathfrak{c}^k:\Hm^k(\calB^{\infty}(X^{\bullet+1};E)^G,\delta^{\bullet})\rightarrow \Hcb^k(G;E) 
\end{equation}
for every $k\geq 0$.

Let $\Gamma<G$ be a lattice. As in the case of representations, given a measurable cocycle $\sigma:\Gamma \times \Omega\rightarrow H$, there exists a natural notion of pullback in bounded cohomology. More precisely, for any Banach $H$-module $E$, the map
\begin{gather}\label{eq_pullback}
\textup{C}_b^\bullet(\sigma):\textup{C}_{cb}^{\bullet}(H;E)^H \rightarrow \textup{C}_{cb}^{\bullet}(\Gamma;E)^\Gamma \ ,
\\
\textup{C}_b^\bullet(\sigma)(\psi)(\gamma_0,\ldots,\gamma_\bullet):=\int_{\Omega} \psi(\sigma(\gamma^{-1}_0,\omega)^{-1},\ldots,\sigma(\gamma^{-1}_\bullet,\omega)^{-1})d\mu(\omega) \ , \nonumber
\end{gather} 
is a well-defined cochain map \cite[Lemma 2.7]{savini:20} inducing a map at the level of cohomology groups
$$
\Hb^k(\sigma):\Hcb^{k}(H;E) \rightarrow \Hb^{k}(\Gamma;E)\ , \ \Hb^k(\sigma)([\psi]):=[\textup{C}_b^k(\sigma)(\psi)] \ . 
$$
for every $k \geq 0$. If $\sigma$ additionally admits a boundary map $\phi:B(G) \times \Omega \rightarrow Y$, one can define
\begin{gather}\label{eq_pullback_boundary}
\textup{C}^\bullet(\Phi^{\Omega}):\calB^\infty(Y^{\bullet+1};E)^H \rightarrow \textup{L}^\infty_{\textup{w}^*}(B(G)^{\bullet+1};E)^\Gamma \ ,
\\ 
\textup{C}^\bullet(\Phi^{\Omega})(\psi)(\xi_0,\ldots,\xi_\bullet):=\int_{\Omega} \psi(\phi(\xi_0,\omega),\ldots,\phi(\xi_\bullet,x))d\mu(\omega) \ . 
\nonumber
\end{gather}

As shown by the second author and Moraschini \cite{moraschini:savini,moraschini:savini:2}, the above map is a cochain map which does not increase the norm and it induces well-defined maps in cohomology 
$$
\textup{H}^k(\Phi^{\Omega}):\textup{H}^k(\calB^\infty(Y^{\bullet+1};E)^H) \rightarrow \textup{H}^k_b(\Gamma;E) \ , \ \textup{H}^k(\Phi^{\Omega})([\psi]):=[\textup{C}^k(\Phi^{\Omega})(\psi)]  
$$
for every $k\geq 0$.

Thanks to \cite[Lemma 2.10]{savini:20}, one can check that the class $\Hb^k(\sigma)([\psi])$ admits as a natural representative the cocycle $\textup{C}^k(\Phi^{\Omega})(\psi)$.

\subsection{The symmetric space $\Simpq$.}\label{section_symmetric_space}
Let $(p,q)\in\matN\times\matN \cup \{\infty\}$ with $p\leq q$ and consider a $(p+q)$-dimensional Hilbert space $\mathcal{H}$ over $\matC$.
Let $\{e_i\}_{i=1}^{p+q}$ be an Hilbert basis for $\calH$. 
We denote by $L(\calH)$ the set of $\matC$-linear bounded operators with respect to the operator norm and by
 $\GL$ the group of bounded invertible $\matC$-linear operators of $\calH$ with bounded inverse. 

We define the Hermitian form $Q$
of signature $(p,q)$ as follows 
\begin{equation*}
 Q(x)= \sum\limits_{i=1}^p x_i\overline{x}_i - \sum\limits_{i= p+1}^{p+q} x_i\overline{x}_i 
\end{equation*}
where $x=\sum\limits_{i=1}^{p+q} x_ie_i$ for all $x\in \calH$.
We denote with $\uppq$ the subgroup of $\GL$ of isometries with respect to $Q$, that means linear maps $h:\calH\rightarrow\calH$ such that $Q(v,w)=Q(h(v),h(w))$ for all $v,w\in\calH$. 
If we define the space
$$\Simpq \coloneqq \left\{V<\calH\;|\;  \dime V=p \; , \; Q_{|V}>0\right\}, $$
then by Witt's theorem the group $\uppq$ acts transitively on it \cite[Theorem 3.9]{artin}.
Moreover, the stabilizer of $V_0\coloneqq\Span \{e_1,\ldots,e_p\}$
is the product $\textup{U}(p)\times \textup{U}(q)$, where $\textup{U}(m)$ is the orthogonal group of the Hilbert space of dimension $m$,
for any $m \in \matN \cup \{\infty\}$.
Hence we can identify $\Simpq$ with the quotient $$\uppq/\textup{U}(p)\times \textup{U}(q)$$
and one can show that it has a structure of simply connected non-positively curved Riemannian symmetric space with real rank $p$ \cite{duchesne:2013}.

Homotheties act trivially on $\Simpq$, so we have an isometric action by isometries of the quotient	
$$\pupq\coloneqq \uppq/\{\lambda \Id\;,\; |\lambda|=1\}$$
on $\Simpq$. 

We define the boundary $\partial\Simpq$ as the set of subspaces of $\calH$ on which the restriction of $Q$ is identically zero (that is \emph{totally isotropic subspaces}).
In $\partial\Simpq$, for each $1\leq k\leq p$, we denote as $\mathcal{I}_k(p,q)$ (or simply $\mathcal{I}_k$) the set of totally isotropic subspaces of dimension $k$. 
In particular, we will be interested in the set $\mathcal{I}_p$ of maximal totally isotropic subspaces. Finally, two points in $\calIpq$ defined by totally isotropic subspaces $V_1$ and $V_2$ are said to be \emph{opposite} if $V_1\cap V_2=0$.

We denote by $Q_{p,q}$ the Hermitian form of signature $(p,q)$ with $p\leq q<+\infty$. 
Let $\{E_i\}_{i=1}^{p+q}$ and $(e_i)_{i\in\matN}$ be two basis of $\matC^{p,q}$ and of an infinite dimensional Hilbert space $\calH$ over $\matC$, respectively.
 
\begin{defn}\label{definition_embedding}
An \emph{embedding} of $\Simpq$ into $\Sim$ is a linear map
$\iota:\matC^{p,q}\rightarrow \calH$ that preserves the Hermitian forms $Q_{p,q}$ and $Q_{p,\infty}$, namely
$Q_{p,\infty}(\iota(x),\iota(y))=Q_{p,q}(x,y)$ for every $x,y\in \matC^{p,q}$.
The group $\uppq$ of linear bounded transformations preserving $Q_{p,q}$ embeds in $\upinf$ in the following way:
the action on $\iota(\matC^{p,q})$ is the one of $\uppq$ and is trivial on the orthogonal complement of $\iota(\matC^{p,q})$. 
\end{defn}

Among all embeddings of $\Simpq$ in $\Sim$, we consider the \emph{standard embedding} defined by the 
map $\iota_0:\matC^{p,q}\rightarrow \calH$ defined as $\iota_0(E_i)=e_i$ for $i=1,\ldots,p+q$.
In this special case, the space $\Simpq$ inside $\Sim$ can be identified with the set
$$\mathcal{V}_0=\{V<\Span\{e_1,\ldots,e_{p+q}\}\; | \;\dime V=p \;,\; Q_{p,\infty|V}>0\}$$
and the group $\uppq$ is identified with elements $g$ in $\upinf$ such that
$$g(e_i)=\sum\limits_{j\in \matN} a_{ij} e_j$$
where, for either $i$ or $j$ bigger than $p+q$, then $a_{ij}=\delta_{ij}$, and the matrix $A=(a_{ij})_{i,j=1}^{p+q}$ represents an element in $\uppq$,
namely it satisfies $$A^* \begin{pmatrix}
                                   \Id_p &0\\
                                   0 & -\Id_q
                                  \end{pmatrix} A=\begin{pmatrix}
                                   \Id_p &0\\
                                   0 & -\Id_q
                                  \end{pmatrix} .$$

The notion of embedding given in Definition \ref{definition_embedding} corresponds to the one of standard embedding given in \cite{duchesne:pozzetti}. 
This choice is motivated by the fact that here we need to distinguish the particular objects described above, whose role among all embeddings is clarified by the following 
\begin{lemma}\label{proposition_embedding}
 Any embedding of $\Simpq$ of $\Sim$ can be obtained by composition of an element $g\in\upinf$  with the standard embedding.
\end{lemma}
\begin{proof}
 Let $\iota:\matC^{p,q}\rightarrow\calH$ an isometric linear map. 
 For each $e_i$ we set $u_i\coloneqq \iota(e_i)$ and 
 $$U_{\iota}\coloneqq \Span\{u_1,\ldots,u_{p+q}\}.$$
There is a natural identification of $\Simpq$ with the subspace of $\Sim$ defined by
 $$\mathcal{V}_{\iota}=\{V<U_{\iota}\; | \;\dime V=p \;,\; Q_{p,\infty|V}>0\}.$$
 If we denote with $U_0$ the subspace of $\calH$ spanned by the first $p+q$ vectors of the basis $(e_i)_{i\in\matN}$,
 we can define an isometric linear map
 $h:U_0\rightarrow U_{\iota}$
 on the basis as follows
 $$h(e_i)=u_i$$
 and then extend it by linearity.
 Since $h$ preserves the Hermitian form $Q$, by Witt's theorem it extends to an isometry of $\calH$ with respect to $Q$, namely to an element $g\in\upinf$. 
 The thesis follows noticing that the isometric linear map $g\circ \iota$ actually gives the standard embedding.
\end{proof}

\begin{oss}\label{remark_embeddings}
As a subspace of the Grassmannian $\textup{Gr}(p+q,\calH)$, the set of embedding of $\calX(p,q)$ inside $\calX(p,\infty)$ naturally inherits the topology induced by principal angles, that in this case coincides with the Wisjman topology (see \cite{duchesne:pozzetti}). 
 Since by Lemma \ref{proposition_embedding} the group $\upinf$ acts transitively on the set of all such embeddings, we have an identification with the $\pup$-orbit of the standard embedding in $\textup{Gr}(p+q,\calH)$. 
Moreover, this can be identified with the quotient $\pup/\Stab_{\pup}V_0$, where $V_0$ is the image of the standard embedding.
\end{oss}

\subsection{The K\"{a}hler class and the Bergman cocycle.}\label{section_khaler}
A crucial difference between the finite case and the infinite one in the context of symmetric spaces is that $\pupq$ is locally compact
for $q<\infty$, whereas $\pup$ is not. 
To overcome this problem we will deal with
the bounded cohomology groups $\Hbul_{\textup{b}}(\pup;\matR)$, namely its continuous bounded
cohomology if we endow $\pup$ with the discrete topology.

Since $\Sim$ is an Hermitian symmetric space, there exists a K\"{a}hler form $\omega$, that is a $\pup$-invariant closed 2-form on $\Sim$. 
Using such an invariant form we can define
$$\omega_x:\pup^3\rightarrow \matR,\;\;\;\;\omega_x(g_0,g_1,g_2)=\frac{1}{\pi}\int_{\Delta(g_0x,g_1x,g_2x)} \omega $$
where $x$ is a point in $\Sim$ and $\Delta(g_0x,g_1x,g_2x)$ is a triangle in $\Sim$ with vertices
$g_0x,g_1x,g_2x$ and geodesic edges. The map $\omega_x$
defines a strict $\pup$-invariant cocycle and by \cite[Lemma 5.3]{duchesne:pozzetti} different choices of the basepoint $\Omega$ lead to cohomologous cocycles.
In this way we obtain a well-defined cohomology class $k^b_{\pup} \in \Hb^2(\pup;\matR)$, called \emph{bounded K\"{a}hler class} of $\pup$.
Now, taking the Gromov norm $||\cdot||_{\infty}$, it follows from the definition that 
\begin{equation}\label{equation_norm_kahler}
||k^b_{\pup}||_{\infty}=\rk\Sim =p. 
\end{equation}

We will need to define the Bergman class extending the one given in finite case, namely to construct a cocycle on the boundary $\calI$.

Given any three maximal totally isotropic subspaces $V_0,V_1,V_2 \in \calI$, since they are contained in a finite dimensional subspace of dimension at most $3p$, 
we can use the definition of the Bergman cocycle associated to $\textup{SU}(p,2p)$ to get a strict $\pup$-invariant cocycle 
$$\beta:\calI^3\rightarrow [-p,p].$$
We recall that the maximal value is taken on triples of pairwise opposite
totally isotropic subspaces which lie in 
a $2p$-dimensional subspace \cite[Proposition 2.1]{Pozzetti}.
Now, given a point $V\in \calI$, the cocycle $C_V$ defined as
$$C_V(g_0,g_1,g_2)=\beta(g_0V,g_1V,g_2V)$$
still represents the bounded K\"{a}hler class $k^b_{\pup}\in \Hb^2(\pup;\matR)$  \cite[Lemma 5.4]{duchesne:pozzetti}.

\subsection{The Toledo invariant.}\label{section_toledo}
Let $\Gamma<\su$ be a complex hyperbolic lattice, $(\Omega,\mu)$ be a standard Borel probability $\Gamma$-space and
let $\sigma:\Gamma \times \Omega\rightarrow \pup$ be a measurable cocycle. 
Following \cite{sarti:savini:superrigidity},
we define the transfer map
$$\Tbdue:\Hb^2(\Gamma;\matR)\rightarrow \Hcb^2(\su;\matR)$$
as the map induced in cohomology by the function
$$\widehat{\Tbdue}:\Linf((\partial\matH)^3;\matR)^{\Gamma}\rightarrow \Linf((\partial\matH)^3;\matR)^{\su} \  ,$$ 
$$\widehat{\Tbdue}(c)(\xi_0,\xi_1,\xi_2)=\int\limits_{\Gamma\backslash\su} c(\bar{g}\xi_0,\bar{g}\xi_1,\bar{g}\xi_2)d\mu_{\Gamma\backslash\su}(\bar{g}) \  .$$
Since $\Hcb^2(\su;\matR)=\matR k^b_{\su}$, where $k^b_{\su}$ is the bounded K\"{a}hler class of $\su$, by composing with the pullback of Equation \eqref{eq_pullback}, we get
\begin{equation}\label{toledo}
 \Tbdue\circ\Hb^2(\sigma) (k^b_{\pup})=\textup{t}_{\sigma}k^b_{\su}
\end{equation}
for some real number $\textup{t}_\sigma$.
\begin{defn}
Given a lattice $\Gamma < \textup{PU}(n,1)$ and a standard Borel $\Gamma$-space $(\Omega,\mu)$, let $\sigma:\Gamma \times \Omega \rightarrow \pup$ be a measurable cocycle. The number $\textup{t}_{\sigma}$ is the \emph{Toledo invariant} associated to $\sigma$.
\end{defn}

Since both $\Tbdue$ and $\Hb^2(\sigma)$ are norm non-increasing, by Equation \eqref{equation_norm_kahler} we have $|\textup{t}_{\sigma}|\leq p$ and the following
\begin{defn}
 Given $\Gamma< \textup{PU}(n,1)$ and a standard Borel $\Gamma$-space $(\Omega,\mu)$,  a measurable cocycle $\sigma:\Gamma \times \Omega \rightarrow \pup$ is \emph{maximal} if its Toledo invariant is equal to $p$. 
\end{defn}

If $\sigma$ admits a boundary map
$\phi: \partial\matH \times \Omega\rightarrow \calI$, exploiting the version of the pull back given by Equation \eqref{eq_pullback_boundary}, we get a map  
$$\Tbdue\circ\Hm^2(\Phi^{\Omega}): \Hm^2 ((\calB^{\infty}(\calI)^{\bullet+1};\matR)^{\pup})\rightarrow \Hb^2(\su;\matR).$$
In this way we can rewrite Equation \eqref{toledo} as follows
$$
\textup{T}_b^2 \circ \text{H}^2_b(\Phi^{\Omega})([\beta])=t_\sigma \kappa^b_{\textup{PU}(n,1)} \ .
$$
Writing down the above equation in terms of cochains, we get the formula
\begin{equation}\label{formula}
 \int\limits_{\Gamma\backslash\su} \int\limits_{\Omega} \beta(\phi(\bar{g}\xi_0,\omega),\phi(\bar{g}\xi_1,\omega),\phi(\bar{g}\xi_2,\omega))d\mu(\omega)d\mu_{\Gamma\backslash \su}(\bar{g})=\textup{t}_{\sigma} \cdot c_n(\xi_0,\xi_1,\xi_2)
\end{equation}
that holds for every triple of distinct points $(\xi_0,\xi_1,\xi_2)$ in $\partial\matH$ \cite{Pozzetti},\cite{sarti:savini:superrigidity}. Here $c_n$ is the Cartan's angular 
invariant that represents the bounded K\"{a}hler class of $\su$ \cite{Goldmancomplex}.

\subsection{Algebraic subgroups of $\GL$.}\label{section_algebraic}
We first introduce the notion of polynomial map.
\begin{defn}
 A map $f:L(\calH)\rightarrow \matR$ is a \emph{polynomial map} if it is a finite sum of maps $f_1,\ldots,f_k$ where
 for each $i=1,\ldots,k$ there exists a $n_i$-linear map $h_i\in L^{n_i}(L(\calH),\matR)$ such that $f_i(g)=h_i(g,\ldots,g)$ for every $g\in L(\calH)$. 
 The \emph{degree} of $f$ is the maximum of the $n_i$'s.
\end{defn}
Now, in parallel to the finite dimensional case, we define an algebraic subgroup as the set of the zero locus of some family of polynomial maps.
\begin{defn}
 A subgroup $G$ of $\GL$ is \emph{algebraic} if there exists a positive integer $n$ and family $\calP$ of polynomial maps of degrees at most $n$
 such that 
 $$G=\{g\in \GL \;|\; P(g,g^{-1})=0 \;,\; \forall P\in\calP\}.$$

 A \emph{strict algebraic subgroup} is a proper algebraic subgroup of $\GL$.
\end{defn}
To define a linear algebraic subgroup of $\textup{GL}(n,\matR)$ we consider polynomial equations in matrix coefficients.
The generalization to infinite dimension of this notion is the content of the following definition.
\begin{defn}[{\cite[Definition 3.4]{duchesne:pozzetti}}]\label{def_polynomial_maps}
 Let $\calH$ be an infinite dimensional Hilbert space and choose an orthonormal basis $(e_n)_{n \in \matN}$. 
 A homogeneous polynomial map $P:L(\calH)\times L(\calH)\rightarrow \matR$ is \emph{standard} of degree $d$ if there exist two naturals $\ell,m$
 such that $\ell+m=d$
 and a family of real coefficients $(\lambda_i)_{i\in\matN^{2\ell}}$ and $(\mu_j)_{j\in\matN^{2m}}$ such that for any $(M,N)\in L(\calH)\times L(\calH)$
 we have that $P$ can be expressed as the absolute convergent series 
 $$P(M,N)=\sum\limits_{i\in\matN^{2\ell},j\in\matN^{2m} } \lambda_i\mu_j P_i(M)P_j(N) $$ where
 $P_i(M)=\prod\limits_{k=0}^{\ell-1} <Me_{i_{2k}},e_{i_{2k+1}}> $ and $P_j(N)=\prod\limits_{k=0}^{m-1} <Ne_{j_{2k}},e_{j_{2k+1}}> $.

 A \emph{standard polynomial map} is a finite sum of standard homogeneous polynomial maps.
 
 An algebraic subgroup of $L(\calH)$ is \emph{standard} if it is defined by a family of standard polynomial maps.
 \end{defn}

Hence we have the following interesting property, that shows how proper standard 
algebraic subgroups are closely related to finite dimensional subspace of $\calH$.
\begin{lemma}\cite[Lemma 3.6]{duchesne:pozzetti}\label{lemma_support}
 If $H$ is a strict standard algebraic group, then there exists a finite dimensional subspace $E$ of $\calH$ such 
 that the the group $H_E\coloneqq\{g\in H \;|\; g(E)=E\;,\; g_{|E^{\bot}}=\id\}$ is a strict algebraic subgroup of $\textup{GL}(E)$.
\end{lemma}
We call the subspace $E$ \emph{support} of the strict algebraic subgroup $H$ and the group $H_E$ the $E$-\emph{part} of $H$.
We are now ready to give the following
\begin{defn}\label{definition_finite_algebraic}
 A \emph{finite dimensional algebraic subgroup} is a standard algebraic subgroup of $\GL$ of the form $H_E$. 
\end{defn}
Hence, it follows by Lemma \ref{lemma_support} a characterization of finite dimensional algebraic subgroups.
\begin{lemma}\label{proposition_algebraic}
 If $E$ is a finite dimensional subspace of $\calH$ and $H$ is a subgroup of $\GL$ contained in $\textup{GL}(E)$, 
 then $H$ is algebraic in $\textup{GL}(E)$ if and only if it is a finite dimensional algebraic subgroup of $\GL$.
\end{lemma}
\begin{proof}
 If $H$ is a finite dimensional algebraic subgroup of $\GL$ then $H=H_E$ and by Lemma \ref{lemma_support} it is algebraic in $\textup{GL}(E)$.
 Conversely, if $H$ is algebraic in $\textup{GL}(E)$, it is also an algebraic subgroup in $\GL$. Moreover, 
 any polynomial which defines $H$ on $\textup{GL}(E)$ can be turned into a polynomial on the entries of the matrices. 
 Hence the same polynomials, seen as standard polynomial maps in the sense of Definition \ref{def_polynomial_maps},
 define a standard algebraic subgroup in $\GL$.
 Since it
 fixes $E^{\perp}$
 then it coincides with its $E$-part and we are done.
\end{proof}

The group $\upinf$ is algebraic subgroup of $\GL$.
Indeed, if $V_0\coloneqq\Span \{e_1,\ldots e_p\}$,
we have that
$$\upinf =\{g\in \GL \;|\; g^{*}\Id_{p,\infty}g=\Id_{p,\infty}\}$$
where $\Id_{p,\infty}$ is the linear map $\Id_{V_0}\oplus -\Id_{V_0^{\perp}}$. Since the 
map $(A,B)\mapsto A^*\Id_{p,\infty}B-\Id_{p,\infty}$ is bilinear on $L(\calH)\times L(\calH)$ then $\upinf$ is algebraic in $\textup{GL}(\calH_{\matR})$
and hence in $\GL$ (see \cite{duchesne:pozzetti} for more details).
By Proposition \ref{proposition_algebraic} we can say immediately that the groups $\uppq$ with $q<\infty$, 
seen as subgroups of $\upinf$ inside $\GL$, are actually finite algebraic since they stabilize the embedding of $\Simpq$ inside $\Sim$. 

Since we work with the quotients $\pupq$ instead of the groups $\uppq$, 
we call \emph{finite algebraic} a subgroup of $\pup$ if its preimage under the projection $\upinf\rightarrow\pup$ is finite algebraic in $\GL$ in the 
sense of Definition \ref{definition_finite_algebraic}.

\subsection{Proof of finite reducibility.}\label{section_proof_reducibility}
Given a measurable cocycle $\sigma:\Gamma \times \Omega\rightarrow \pup$, one can ask when its image is contained in some suitable subgroup of $\pup$.

\begin{defn}\label{def_finite_red}
 A cocycle $\sigma:\Gamma \times \Omega\rightarrow \pup$ is \emph{finitely reducible} if it admits a cohomologous cocycle
 with image contained in a finite dimensional algebraic subgroup of $\pup$.
\end{defn}

Before proving the main theorem, we recall that a $p$-\emph{chain} is a copy of $\mathcal{I}_p(p,p)$ in $\mathcal{I}_p(p,\infty)$ determined by an embedding of $\mathcal{X}(p,p)$ in $\mathcal{X}(p,\infty)$. 

\begin{defn}
 A measurable map $\phi: \partial\matH\rightarrow \calI$ \emph{almost surely maps chains to chains} if for almost every
chain $\calC \subset \partial \matH$ there is a $p$-chain $\calT\subset \calI$ such that for almost every point $\xi \in \calC$,
$\phi(\xi)\in \calT$.
\end{defn}

An equivalent condition \cite[Lemma 4.2]{Pozzetti} to the one above is to check that, 
for almost every pair $(x,y)\in \partial\matH\times\partial\matH$, the points $\phi(\xi_0)$ and $\phi(\xi_1)$ are opposite and, 
for almost every $z\in \calC_{\xi_0,\xi_1}$, the subspace $\phi(z)$ is contained in $\langle \phi(\xi_0),\phi(\xi_1)\rangle$.
Before passing to the proof of Theorem \ref{theorem_reducibility}, we need the following result about maps that almost surely maps chains to chains, which is a slight refinement of \cite[Proposition 6.2]{duchesne:pozzetti}. 
Since there is a natural embedding $\partial\matH\subset \mathbb{P}^n\mathbb{C}$, we can say that a set of $k\leq n+1$ points in $\partial\matH$ is \emph{generic} 
if, for every $1<h\leq k$, any subset of $h$ points does not span a $(h-2)$-dimensional subspace.

\begin{lemma}\label{lemma_unicity}
Let $\phi: \partial\matH\rightarrow \calI$ be a measurable map that almost surely maps chains to chains.
Then there exists a unique minimal totally geodesic embedded copy of $\calX(p,q)\subset\calX(p,\infty)$ that contains the image of almost every $(n+1)$-tuple of generic points in $\partial\matH$. Moreover, $p\leq q\leq np$.
\end{lemma}

\begin{proof}
We argue by induction on $n$. The case $n=1$ is clear, since there is only one chain $\calC$ in $\partial\mathbb{H}^1_{\matC}$ 
and for almost every $\eta_1,\eta_2\in \calC$ the subspace $\langle \phi(\eta_1),\phi(\eta_2)\rangle\subset \calH$ defines a copy of $\calX(p,p)\subset \calX(p,\infty)$. The fact that $\phi$ almost surely maps chains to chains implies that for almost every $\xi$ in $\partial\matH$ we have $\phi(\xi)< \langle \phi(\eta_1),\phi(\eta_2)\rangle$. 

Assume that the statement holds for $n-1$.
Thanks to the construction in \cite[Section 7.1]{duchesne:pozzetti}, we can define a full-measure subset $\calG$ of the set of $(n+1)$-tuple of points in general position of $\partial\matH$ such that for every $(\xi_0,\ldots,\xi_n)\in \calG$ the following conditions hold:
\begin{itemize}
\item $\phi_{|\langle \xi_0,\ldots,\xi_{n-1}\rangle}$ almost surely maps chains to chains;
\item for almost every $\eta \in\langle\xi_{0},\ldots,\xi_{n-1}\rangle$ then $\langle\phi(\eta),\phi(\xi_{n-1})\rangle$ is a $2p$-dimensional subspace on which the restriction of $Q$ has signature $(p,p)$;
\item for almost every $\eta \in\langle\xi_{n-1},\xi_n\rangle$ then $\langle\phi(\eta),\phi(\xi_{n-1})\rangle$ is a $2p$-dimensional subspace on which the restriction of $Q$ has signature $(p,p)$;
\item for almost every $\eta_1 \in\langle\xi_{n-1},\xi_n\rangle,\,\eta_2 \in\langle\xi_{0},\ldots,\xi_{n-1}\rangle$ the space $\langle\phi(\eta_1),\phi(\eta_2)\rangle$ has dimension $2p$ and the restriction of $Q$ has signature $(p,p)$.
\end{itemize}

As proved in \cite[Proposition 6.2]{duchesne:pozzetti}, for almost every $(\xi_0,\ldots,\xi_n)\in \calG$ the space
$$V_{\xi_0,\ldots,\xi_n}\coloneqq \langle\phi(\xi_{0}),\ldots,\phi(\xi_n)\rangle$$
contains $\phi(\eta)$ for almost every $\eta \in \partial\matH$. 
Furthermore, the restriction of $Q$ to $V_{\xi_0,\ldots,\xi_n}$ is non-degenerate of signature $(p,q)$ with $p\leq q\leq np$.

We now prove that almost every pair of tuple $((\xi_0,\ldots,\xi_n),(\eta_0,\ldots,\eta_n)) \in \calG^2$ give the same subspace. 
We first note that, since $V_{\xi_0,\ldots,\xi_n}$ contains the image of almost every point in $\partial\matH$, it clearly contains $\phi(\eta_{0}),\ldots,\phi(\eta_n)$, and hence $\langle\phi(\eta_{0}),\ldots,\phi(\eta_n)\rangle$, for almost every $(\eta_0,\ldots,\eta_n)\in \calG$. 
Hence there exists a full-measure subset $\calQ\subset \calG\times \calG$ such that
$$V_{\xi_0,\ldots,\xi_n}<V_{\eta_0,\ldots,\eta_n}$$
 for almost every $((\xi_0,\ldots,\xi_n),(\eta_0,\ldots,\eta_n))\in \calQ$.
By taking the measure-preserving idempotent function of $\calG\times \calG$ which swap the tuple, one gets a second full-measure subsets $\overline{\calQ}$.
Hence the intersection
$\calQ\cap \overline{\calQ}$ is a full-measure subset of $\calG\times \calG$ of pairs $(\xi_0,\ldots,\xi_n),(\eta_0,\ldots,\eta_n)$ such that
$$V_{\xi_0,\ldots,\xi_n}=V_{\eta_0,\ldots,\eta_n}\,,$$
which implies the uniqueness.
 
 A similar argument can be used to prove minimality, namely that every linear subspace $W<\calH$ containing the image of a full-measure subset of $\partial\matH$ must contain the spaces constructed above. 
\end{proof}

\begin{oss}
It seems natural to investigate the effective dimension of the copy of $\partial\calX(p,q)$ which contains the essential image of $\phi$ provided by Lemma \ref{lemma_unicity}.
For instance, given a chain preserving map $\psi:\partial\matH\rightarrow \partial\mathbb{H}_{\matC}^p$, Burger and Iozzi \cite{BIcartan} proved the following dichotomy: if the image of almost every triple $(\xi_0,\xi_1,\xi_2)$ of generic points is generic as well, then $\psi$ coincides almost everywhere with the map induced on boundaries by an isometric holomorphic embedding $\matH\rightarrow \mathbb{H}_{\matC}^p$. If not, then the image is essentially contained into a chain in $\partial\mathbb{H}^p$.  

In our more general context, we do not know if such a dichotomy holds.
However,  in our setting, the two cases described above can be interpreted as the limit cases as follows.
In fact, if $\phi:\partial\matH\rightarrow \calI$ as in Lemma \ref{lemma_unicity} sends almost every $(n+1)$-tuple of generic points of $\partial\matH$ to $(n+1)$ generic points of $\calI$, then we have that the essential image of $\phi$ is contained in $\partial\calX(p,np)$.
On the other hand, by the same argument used in \cite[Proposition 2.2]{BIcartan}, if there is a positive measure subset of triple in $(\partial\matH)^3$ not on a chain whose image lies on a chain, then the image of $\phi$ is essentially contained into one copy of $\partial\calX(p,p)$. 
We point out that this two cases do not produce a dichotomy, but a characterization of the cases when $q=p$ and $q=np$ in the notation of Lemma \ref{lemma_unicity}. 
\end{oss}

Now we are ready to give the proof of the main result of this part. 
\begin{proof}[Proof of Theorem \ref{theorem_reducibility}]
By the Equation \eqref{formula} and using \cite[Corollary 6.1]{duchesne:pozzetti}, it follows that 
almost every slice $\phi_{\omega}$ almost maps chains to chains. 
Hence, by Lemma \ref{lemma_unicity}, for almost every $\omega\in \Omega$ 
there exists a unique minimal totally geodesic embedding $\mathcal{X}_{\omega}(p,q_{\omega})\subset \Sim$ such that
$\EssIma(\phi_{\omega})\subset \partial\mathcal{X}_{\omega}(p,q_{\omega})$ for some $p\leq q_{\omega}\leq np$. 
Notice the $\mathcal{X}_\omega$ depends only on the measure class of $\phi_\omega$. By \cite[Chapter VII, Lemma 1.3]{margulis:libro} the function
$$
\Phi:\Omega \rightarrow \mathrm{L}^0(\Omega, \mathcal{I}_p) \ , \ \ \ \Phi(\omega):=\phi_\omega
$$
is measurable, where $\mathrm{L}^0(\Omega, \mathcal{I}_p)$ has again the measurable structure coming from the topology of convergence in measure (the distance on $\mathcal{I}_p$ is the one of Remark \ref{remark_embeddings}). By Fubini's Theorem, there exist $\xi_0,\ldots,\xi_n \in \partial \mathbb{H}^n_{\mathbb{C}}$ so that $\mathcal{X}_\omega(p,q_\omega):=\langle \phi_\omega(\xi_0),\ldots \phi_\omega(\xi_n) \rangle$ for almost every $\omega \in \Omega$. Since $\phi_\omega$ depends measurably on $\omega$, the same holds for $\mathcal{X}(p,q_\omega)$.

Moreover, the equivariance of $\phi$ implies that 
$$\sigma(\gamma,\omega)\mathcal{X}_{\omega}(p,q_{\omega})=\calX_{\gamma \omega}(p,q_{\gamma \omega})$$
for almost every $\gamma\in \Gamma$ and $\omega\in \Omega$. By the ergodicity of $\Omega$, 
the number $q_\omega$ is essentially constant, namely $q_{\omega}=q$ for almost every $\omega\in \Omega$.
If we denote by $\iota_{\omega}$ the isometric linear 
map that induces the embedding $\mathcal{X}_{\omega}(p,q)\subset \Sim$,
the uniqueness of $\mathcal{X}_{\omega}(p,q)$, together with the $\sigma$-equivariance of $\phi$, implies that the map 
\begin{equation}\label{equation_map}
\Omega\rightarrow  \pup/\Stab_{\pup}(V_0)\,,\;\;\;\; \omega\mapsto \mathcal{X}_{\omega}(p,q)
\end{equation}
is measurable (with respect to the measurable structure discussed in Remark \ref{remark_embeddings}) and $\sigma$-equivariant.
Here $\Stab_{\pup}V_0$ is the subgroup of $\pup$ preserving the subspace $V_0$.
Now, thanks to the differentiable structure of the group $\pup$, we can compose the function in Equation \eqref{equation_map} with a measurable section 
$$\pup/\Stab_{\pup}(V_0)\rightarrow \pup$$ in order to obtain a measurable map
$$f:\Omega\rightarrow \pup,\;\;\;\; f(\omega)=g_{\omega}^{-1}.$$
By construction, $f(\omega)$ sends $\calX_{\omega}(p,q)$ to the standard embedded copy $\calX(p,q) \subset \calX(p,\infty)$.

We consider the twisted cocycle 
$\sigma^f:\Gamma \times \Omega\rightarrow \pup$ defined as
$$\sigma^f(\gamma,\omega)\coloneqq f(\gamma \omega)^{-1}\sigma (\gamma,\omega)f(\omega)$$
and the associated twisted boundary map 
$\phi^f:\partial\matH \times \Omega\rightarrow \calI$
defined by
$$\phi^f(\xi,\omega)\coloneqq f(\omega)^{-1} \phi(\xi,\omega)\ .$$
Now, by definition of $f$, for almost every $\omega\in \Omega$ the image of almost every slice $\phi_{\omega}$ is contained in the boundary of a fixed
$\mathcal{X}(p,q)$. For almost every $\omega\in \Omega$, denote by
$E_{\omega}$ the full measure set of points $\xi$ in $\partial\matH$ such that $\phi^f_{\omega}(\xi)\in \partial\mathcal{X}(p,q)$.
Consider now the set $E=\bigcup\limits_{\omega\in \Omega} E_{\omega}\times\{\omega\}$ (that is of full measure in $\partial\matH\times \Omega$,
by Fubini's theorem) and the diagonal action of $\Gamma$ given by
$$\gamma (\xi,\omega) = (\gamma\xi, \gamma \omega).$$
Since $\Gamma$ is countable, we find an invariant full measure subset
$\overline{E}$ such that $\phi^f(\overline{E})\subset \partial\mathcal{X}(p,q)$. More precisely, we set
$$\overline{E}=\bigcap\limits_{\gamma\in \Gamma} \gamma E \ ,$$
where $\gamma$ acts diagonally. Being the countable intersection of full measure sets, it is clear that $\overline{E}$ has full measure. 
Now, since the image of a full measure set under $\phi^f$ is contained in the boundary of the embedded $\mathcal{X}(p,q)$, it follows that 
the image of the twisted cocycle $\sigma^f$ is contained in $\Stab_{\pup}V_0$, which is finite algebraic as desired.
\end{proof}

\begin{oss}
 The descending chain condition that holds for Noetherian spaces (as algebraic groups are), allows to define the algebraic hull for cocycles into algebraic groups.
 This can not be adapted for $\pup$, namely there exits no well-defined minimal strict algebraic group containing the image of a twisted cocycle. 
 Nevertheless, by Theorem \ref{theorem_reducibility}, any maximal cocycles has a representative in its cohomology class 
whose image is contained into the embedding of $\pupq$ in $\pup$, which is algebraic.
For such particular measurable cocycles, our result recover an algebraic flavour.
\end{oss}

\section{Consequences of finite reducibility}\label{section_consequences}
 The aim of this last section is to link Theorem \ref{theorem_boundary_map} and Theorem \ref{theorem_reducibility}.
We consider the setting of Theorem \ref{theorem_reducibility}, namely $\Gamma$ is a complex hyperbolic lattice, $(\Omega,\mu)$ is an ergodic standard Borel probability $\Gamma$-space and $\sigma:\Gamma \times \Omega\rightarrow \pup$ is a maximal cocycle. 
If we assume that $\sigma$ is non-elementary, Theorem \ref{theorem_boundary_map} provides a boundary map $\phi:\partial\matH\times \Omega\rightarrow \partial\calX(p,\infty)$. 
Moreover, by Remark \ref{remark_boundary_fiinite_dimensional} such a map takes values into $\calIk$ for some $k\leq p$.
Unfortunately, this is not sufficient to prove reducibility as in Theorem \ref{theorem_reducibility}, since such $k$ might be strictly less then $p$.

However, for cocycles $\sigma:\Gamma \times \Omega\rightarrow \puoneinf$ one can exploit the geometry of $\calX(1,\infty)=\mathbb{H}^{\infty}_{\matC}$ and of its boundary to prove Theorem \ref{corollary_mostow}.

\begin{proof}[Proof of Theorem \ref{corollary_mostow}]
We first show that maximal cocycles cannot be non-elementary. 
In fact, by ergodicity, a $\sigma$-equivariant family of flats can be made of points or lines.
In both cases one can twist $\sigma$ into a cocycles whose image is contained either in the stabilizer of a point or a geodesic, which are both amenable. Since amenable groups have trivial bounded cohomology, we have a contradiction to maximality.

Since $\sigma$ is not elementary, Theorem \ref{theorem_boundary_map} provides a boundary map $\partial\matH\times \Omega\rightarrow \partial\mathbb{H}^{\infty}_{\matC}$ and then we can apply Theorem \ref{theorem_reducibility}. Hence we have that $\sigma$ is cohomologous to a cocycle $\widetilde{\sigma}$ whose image is contained in the stabilizer of an embedded copy of $\matH$ in $\mathbb{H}_{\matC}^{\infty}$. 
The stabilizer $\Stab_{\puoneinf}(\matH)$ is an almost direct product with one factor isomorphic to $\su$. By composing with the projection on such factor we get a maximal cocycle. Hence we can apply \cite[Theorem 1.5]{moraschini:savini:2} and we are done.
\end{proof}

\begin{oss}
In the general setting of Theorem \ref{theorem_reducibility}, as pointed out in Remark \ref{remark_boundary_fiinite_dimensional}, Theorem \ref{theorem_reducibility} provides a boundary map into some $\calIk$. 
In \cite{duchesne:pozzetti} the authors exploited Proposition \ref{proposition:bdl} to rule out the case $k<p$ for Zariski dense representations. 
In the tentative to adapt such argument in the context of cocycles, we stuck in the final part. Precisely, following the proof of \cite[Theorem 1.7]{duchesne:pozzetti}, one can construct a $\sigma$-equivariant family $\{W_{\omega}\}_{\omega\in \Omega}$ of non-trivial subspaces of $\Lambda^d\calH$ for some $d$. 
Since the stabilizer of such spaces are standard algebraic subgroups, it would be enough to twist the cocycle in order to get a cocycle with image contained in one of this stabilizers.
However, the action of $\pup$ on the subspaces (a priori of infinite dimension) of $\Lambda^d\calH$ seems to us quite mysterious. Even before, one should clarifies the measurable structures involved.
To conclude as in the proof of Theorem \ref{theorem_reducibility} or \cite[Theorem 2]{sarti:savini:superrigidity} one should identify the $\pup$-orbit of some $W_{\omega}$ with the quotient $\pup/\Stab_{\pup} W_{\omega}$, for instance by proving that the action is smooth, which is also not clear to us. 
\end{oss}

\bibliographystyle{amsalpha}

\bibliography{biblionote}

\end{document}